\documentclass[11pt]{article}
\usepackage[margin=.8in,left=.8in]{geometry}
\usepackage{amsmath}
\usepackage{amsfonts}
\usepackage{amssymb}
\usepackage{amsthm}
\usepackage{mathtools} 
\usepackage{subcaption}
\usepackage{float}
\usepackage{graphicx}
\usepackage{color}
\usepackage{xcolor}
\usepackage{xfrac}
\usepackage{stmaryrd}
\usepackage[utf8]{inputenc}
\usepackage{hyperref}
\usepackage[all]{xy}

\usepackage{tikz}
\usepackage{tikz-cd}
\usepackage{lipsum}
\usepackage{adjustbox}
\usetikzlibrary{decorations.pathmorphing,shapes}

\usepackage[shortlabels]{enumitem}

\usepackage{multirow}

\usepackage{stmaryrd}    

\usepackage{helvet}   

\usepackage{color, colortbl}
\definecolor{Gray}{gray}{0.9}
\definecolor{lightgray}{rgb}{0.9,0.9,0.9}

\definecolor{darkblue}{rgb}{0.05,0.25,0.65}
\definecolor{greenii}{RGB}{20,140,10}

\definecolor{orangeii}{RGB}{200,100,5}

\usepackage{amsmath}

\usepackage{floatflt}  
\usepackage{wrapfig} 
\usepackage{array}     
\newcolumntype{L}[1]{>{\raggedright\let\newline\\\arraybackslash\hspace{0pt}}m{#1}}
\newcolumntype{C}[1]{>{\centering\let\newline\\\arraybackslash\hspace{0pt}}m{#1}}
\newcolumntype{R}[1]{>{\raggedleft\let\newline\\\arraybackslash\hspace{0pt}}m{#1}}

\usepackage{diagbox}  

\makeatletter
\newcommand{\raisemath}[1]{\mathpalette{\raisem@th{#1}}}
\newcommand{\raisem@th}[3]{\raisebox{#1}{$#2#3$}}
\makeatother

\usepackage{tabularx}   

\usepackage[new]{old-arrows}   

\newdir{> }{{}*!/10pt/@{>}}

\newcommand{\dslash}{/\!\!/}

\makeatletter
\newif\if@sup
\newtoks\@sups
\def\append@sup#1{\edef\act{\noexpand\@sups={\the\@sups #1}}\act}%
\def\reset@sup{\@supfalse\@sups={}}%
\def\mk@scripts#1#2{\if #2/ \if@sup ^{\the\@sups}\fi \else%
  \ifx #1_ \if@sup ^{\the\@sups}\reset@sup \fi {}_{#2}%
  \else \append@sup#2 \@suptrue \fi%
  \expandafter\mk@scripts\fi}
\def\tensor#1#2{\reset@sup#1\mk@scripts#2_/}
\def\multiscripts#1#2#3{\reset@sup{}\mk@scripts#1_/#2%
  \reset@sup\mk@scripts#3_/}
\makeatother

\makeatletter
\newbox\slashbox \setbox\slashbox=\hbox{$/$}
\def\itex@pslash#1{\setbox\@tempboxa=\hbox{$#1$}
  \@tempdima=0.5\wd\slashbox \advance\@tempdima 0.5\wd\@tempboxa
  \copy\slashbox \kern-\@tempdima \box\@tempboxa}
\def\slash{\protect\itex@pslash}
\makeatother

\def\clap#1{\hbox to 0pt{\hss#1\hss}}

\def\mathclap{\mathpalette\mathclapinternal}

\def\mathclapinternal#1#2{\clap{$\mathsurround=0pt#1{#2}$}}

\let\oldroot\root
\def\root#1#2{\oldroot #1 \of{#2}}
\renewcommand{\sqrt}[2][]{\oldroot #1 \of{#2}}

\DeclareSymbolFont{symbolsC}{U}{txsyc}{m}{n}
\SetSymbolFont{symbolsC}{bold}{U}{txsyc}{bx}{n}
\DeclareFontSubstitution{U}{txsyc}{m}{n}

\DeclareSymbolFont{stmry}{U}{stmry}{m}{n}
\SetSymbolFont{stmry}{bold}{U}{stmry}{b}{n}

\DeclareFontFamily{OMX}{MnSymbolE}{}
\DeclareSymbolFont{mnomx}{OMX}{MnSymbolE}{m}{n}
\SetSymbolFont{mnomx}{bold}{OMX}{MnSymbolE}{b}{n}
\DeclareFontShape{OMX}{MnSymbolE}{m}{n}{
    <-6>  MnSymbolE5
   <6-7>  MnSymbolE6
   <7-8>  MnSymbolE7
   <8-9>  MnSymbolE8
   <9-10> MnSymbolE9
  <10-12> MnSymbolE10
  <12->   MnSymbolE12}{}

\makeatletter
\def\Decl@Mn@Delim#1#2#3#4{%
  \if\relax\noexpand#1%
    \let#1\undefined
  \fi
  \DeclareMathDelimiter{#1}{#2}{#3}{#4}{#3}{#4}}
\def\Decl@Mn@Open#1#2#3{\Decl@Mn@Delim{#1}{\mathopen}{#2}{#3}}
\def\Decl@Mn@Close#1#2#3{\Decl@Mn@Delim{#1}{\mathclose}{#2}{#3}}
\Decl@Mn@Open{\llangle}{mnomx}{'164}
\Decl@Mn@Close{\rrangle}{mnomx}{'171}
\Decl@Mn@Open{\lmoustache}{mnomx}{'245}
\Decl@Mn@Close{\rmoustache}{mnomx}{'244}
\makeatother

\newcommand{\mathfr}{\mathfrak}

\def\co{\colon\thinspace}
\usepackage{colonequals}
\def\coeq{\colonequals}
\newcommand\noloc{%
  \nobreak
  \mspace{6mu plus 1mu}
  {:}
  \nonscript\mkern-\thinmuskip
  \mathpunct{}
  \mspace{2mu}
}

\newcommand{\g}{\mathfrak{g}}
\newcommand{\h}{\mathfrak{h}}

\newcommand{\R}{\ensuremath{\mathbb R}}

\newcommand{\Q}{\ensuremath{\mathbb Q}}

\renewcommand{\(}{\begin{equation}}
\renewcommand{\)}{\end{equation}}
\newcommand{\bea}{\begin{eqnarray*}}
\newcommand{\eea}{\end{eqnarray*}}

\newcommand{\mc}[1]{\mathcal{#1}}
\newcommand{\abs}[1]{\lvert #1 \rvert}

\def\CC{\mathbb{C}}
\def\GG{\mathbb{G}}

\def\QQ{\mathbb{Q}}
\def\RR{\mathbb{R}}
\def\ZZ{\mathbb{Z}}

\DeclareMathOperator{\Aut}{Aut}

\DeclareMathOperator{\Der}{Der}

\DeclareMathOperator{\Hom}{Hom}
\DeclareMathOperator{\iHom}{hom}
\DeclareMathOperator{\id}{id}

\DeclareMathOperator{\Map}{Map}

\DeclareMathOperator{\rank}{rank}

\DeclareMathOperator{\DGCA}{DGCA}
\DeclareMathOperator{\sfree}{sf}
\DeclareMathOperator{\sgn}{sgn}
\DeclareMathOperator{\Trd}{Trd}
\DeclareMathOperator{\Tot}{Tot}

\usepackage{cleveref}

\crefformat{section}{\S#2#1#3} 
\crefformat{subsection}{\S#2#1#3}
\crefformat{subsubsection}{\S#2#1#3}

\newtheorem{theorem}{Theorem}[section]

\newtheorem{prop}[theorem]{Proposition}
\newtheorem{cor}[theorem]{Corollary}

\theoremstyle{definition}
\newtheorem{construction}[theorem]{Construction}
\newtheorem{defn}[theorem]{Definition}

\newtheorem{example}[theorem]{Example}

\newtheorem{remark}[theorem]{Remark}
\newtheorem{note[theorem]}{Note}

\newcommand{\proofstep}[1]{
  \mbox{\small #1}
}

\newcommand{\eps}{\varepsilon}
\newcommand{\p}{\mathfrak{p}}

\usepackage{amsfonts}

\begin{document}

\title{Mysterious Triality and the Exceptional Symmetry of Loop Spaces}
\author{Hisham Sati, \; Alexander A. Voronov}
\date{}                                           

\maketitle

\begin{abstract}
In previous work \cite{SV1}\cite{SV2}, we introduced Mysterious Triality, extending 
the Mysterious Duality \cite{INV} between physics and algebraic geometry to include algebraic topology
in the form of rational homotopy theory. 
Starting with the rational Sullivan minimal model of the 4-sphere $S^4$, capturing the dynamics 
of M-theory  via  Hypothesis H, this progresses to the dimensional reduction of M-theory on torus 
$T^k$, $k \ge 1$, with its dynamics described 
via the iterated cyclic loop space  $\mathcal{L}_c^k S^4$ of the 4-sphere. 
From this, we also extracted data corresponding to the maximal torus/Cartan subalgebra and the Weyl group of the exceptional Lie group/algebra of type $E_k$. 

\smallskip
In this paper, we discover much richer symmetry by extending the data of 
the Cartan subalgebra to a maximal parabolic subalgebra $\mathfrak{p}_k^{k(k)}$ of the split real form $\mathfrak{e}_{k(k)}$ of the exceptional Lie algebra of type $E_k$  by exhibiting an action, in rational homotopy category, of $\mathfrak{p}_k^{k(k)}$ on
the slightly more symmetric than $\mc{L}_c^k S^4$ toroidification $\mc{T}^k S^4$.
This action universally represents symmetries of the equations of motion of supergravity in the reduction of M-theory to $11-k$ dimensions.

\smallskip 
Along the way, we identify the minimal model of the toroidification $\mathcal{T}^k S^4$,
generalizing the results of 
Vigu\'{e}-Poirrier, Sullivan, and Burghelea, and establish an algebraic toroidification/totalization adjunction.

 \end{abstract}

\medskip

\tableofcontents

\section{Introduction} 

 The occurrence of the global exceptional $E_k$ symmetries in toroidally-compactified to $D = {11-k}$ dimensions
 eleven-dimensional supergravity was conjectured by Cremmer and Julia in \cite{CJ2}\cite{CJ}\cite{CJ1}, based
on the structure of scalar fields in lower dimensions. 
This has been demonstrated for 
$D=9$ \cite{BHO95}, $D=4$ \cite{CJ}, and $D=3$ \cite{Jul1}\cite{Jul2}\cite{Miz},
and in general for $D \geq 3$ in \cite{CJLP1}, who prove that the scalar
Lagrangians have global $E_{k}$ symmetries, after rescaling and after dualizing all the
$(D-2)$-form potentials
to give rise to the maximal number of scalar fields.  
See e.g.\ \cite{Sam} for a recent survey.

\medskip 
The resulting moduli space of the corresponding putative quantum/high energy theory, M-theory, compactified on 
(reduced on, or simply on) $T^k$ is 
 usually taken to be the double quotient $G(\ZZ) \backslash G / K$, where $G$ is the U-duality group, 
 which is (the split real form of) the Lie group $E_k$, $G(\ZZ)$ is its integral form, and $K$ is 
 the maximal compact subgroup of $G$ \cite{HT}\cite{OP}. 
 In  the context of mysterious duality \cite{INV}, for ``flat tori with no C-field,''  the moduli space was taken to be of the form
$
\mathcal{M}_k := A/W,
$
where $A$ is the maximal split torus of $G$ and $W$ is its Weyl group.

\medskip 
This simplification of the moduli space, as we explained in \cite{SV1}\cite{SV2}, can be thought of as abelianization. 
In this paper, we provide the other direction, which
can be viewed as lifting to the full exceptional group via the Iwasawa decomposition $G=KAN$
(and also through a maximal parabolic):
\vspace{-1mm}
\(
\begin{tikzcd}[row sep=-2pt]
\label{nonabelianization} 
 \fbox{$\mathcal{M}_k = A/W$}
\arrow[rrr, rightsquigarrow, "{\bf \color{darkblue} ``unabelianize"}"]
&&& \fcolorbox{black}{lightgray}{$K \backslash G / G(\ZZ) = (AN) /G(\ZZ)$}
\\
\text{Abelian/Cartan} &&& \text{Nonabelian/full Lie group}
\end{tikzcd} 
\)
Item {\bf (vii)} in the outlook list in \cite{SV2} is \emph{Probing into the nonabelian part of the moduli space}:
We are, in a sense, supplying the topological content of the 
nonabelian unipotent part $N$ of \eqref{nonabelianization} by considering fields which carry 
nontrivial weights of the toroidal part $A$. This is the task that we take up in this paper. 
Here we will work locally, at the level of Lie algebras.
We will consider the global 
perspective at the level of Lie groups in \cite{SV4}. 

\medskip 
The starting point is Hypothesis H that the Sullivan minimal model of the 4-sphere $S^4$ in rational homotopy theory
captures the dynamics of the fields in M-theory, as originally proposed in \cite{Sati13}, and developed further
in \cite{FSS17}\cite{FSS19b} \cite{FSS-WZW}\cite{GS21}\cite{SV1}\cite{SV2}.

\paragraph{The rational models.} 
We are interested in generalizing constructions of loop spaces. We considered 
two directions in \cite{SV1}\cite{SV2}: the first is iterated cyclic loop spaces, which we highlighted 
extensively there, and the second is toroidification spaces, i.e., 
mapping directly from the tori, which we will emphasize here. 
Let $\mc{L} Z \coeq \Map^0 (S^1, Z)$ be the \emph{free loop space} of a
topological space $Z$, which we will assume to be path-connected. Here $\Map^0$ stands for the path component of the constant map.
The free loop space $\mc{L} Z$ admits a natural action of the group $S^1$ by rotating loops, and we define the \emph{cyclic loop space} or \emph{cyclification} of $Z$ as the \emph{homotopy quotient}
\vspace{-3mm} 
\[
\mathcal{L}_c Z \coeq \mc{L} Z \dslash S^1
\]

\vspace{-2mm} 
\noindent and the \emph{iterated cyclic loop space} recursively as
\[
\mathcal{L}_c^k Z \coeq \mc{L}_c( \mc{L}_c^{k-1} Z ).
\]
As in the case $k=1$, the \emph{iterated 
free loop space} $\mc{L}^k Z = \Map^0 (T^k, Z)$ admits a natural (right) action of the group $T^k$ by
rotating loops
 and we define the
 \emph{toroidification} to be the homotopy quotient
 \vspace{-2mm} 
\begin{equation}
\label{hquotient}
 \mc{T}^k Z := \mc{L}^k Z \dslash T^k\;.
\end{equation}

\vspace{-2mm} 
\noindent The homotopy quotient may be defined using the Borel construction
$\mc{L}^k Z \times_{T^k} ET^k$
as the quotient of $\mc{L}^k Z \times ET^k$ by the (anti)diagonal action of $T^k$, where $ET^k$ is the universal $T^k$-bundle. Note the matching $\mc{T}^1 Z = \mc{L}_c Z$ for $k=1$.
 As in \cite{SV1}\cite{SV2}, we use \emph{real}
  Sullivan minimal models, in the context of rational
  homotopy theory over the reals (see \cite{BSzcz}\cite{GM13}\linebreak[0]\cite{FSS-Chern}).

\paragraph{Relation to cyclic homology.}
\!\!\! For a path-connected space $X$, the homology of the Borel construction
$\mc{T}^1 X = \Map (S^1, X) \linebreak[0] \times_{S^1} ES^1 $, with coefficients in a commutative unital ring $R$, can be identified with the cyclic
homology of the singular chains $C_*(\Omega X; R)$ on the based (Moore) loop space
\cite{Good}\cite{BF86}. 
If $X$ is simply connected and of \emph{finite type}, i.e.,  $\dim \pi_i(X)\otimes \mathbb{Q}$ is finite for each $i$,
the Sullivan minimal model for the above Borel construction 
computes the cyclic homology of the
singular cochain algebra $C^*(X; \Q)$,
allowing for an explicit calculation of the cyclic homology of
the space $X$ \cite{vigue-burghelea}. 
For a simply connected space $X$, the cohomology
of $\mc{T}^1 X$ 
with coefficients in $R$ is dual to the
cyclic cohomology of $X$
as a module over $H^*(BS^1; R)=R[u]$, with $|u|=2$ 
\cite{Jones}. 
When $R=\mathbb{Q}$, the cyclic homology of $C^*(X; \mathbb{Q})$
is isomorphic to that of the differential graded (dg-) commutative algebra (DGCA) $A_{\rm PL}(X)$ of rational polynomial forms on $X$.
An iterated cyclic homology of connected rational DGCAs as a generalization of the 
ordinary cyclic homology is introduced in \cite{KY07} and shown to be 
isomorphic to the rational cohomology algebra of the Borel space
$\Map(T^k, X) \times_{S^1} ES^1$ for a $k$-connected $X$.  
While this seems close to our perspective
here and in \cite{SV1}\cite{SV2}, there are two main differences: in \cite{KY07},
$X$ is taken to be $k$-connected, and the Borel space involves a mix of the circle
and torus, with the circle action on the mapping space
induced by the diagonal action on the source $T^k$.

\paragraph{Nilpotency.} 
From \cite{HMRS78}, each component of the iterated free loop space 
$\Map(T^k, S^4)$ is nilpotent since the codomain space is. 
Furthermore, for the path component $\Map^f(T^k, S^4)$ containing 
a map $f: T^k \to S^4$, the rationalization is given by 
$\Map^f(T^k, S^4)_\QQ= \Map^{rf}(T^k, S^4_\QQ)$, where 
$r:S^4 \to S^4_\QQ$ is the rationalization map. The rational models can 
depend on the component and are classified essentially by the cohomology of the 
domain, in our case $T^k$ (see \cite{MR85}\cite{BS}).\footnote{Already for self-maps $f$ of the 4-sphere,
one has via \cite{MR85} different identificaion of the components of the mapping spaces: ${\Map}^{f=0}\big(S^4_\QQ, S^4_\QQ\big)\simeq S^4_\QQ \times S^3_\QQ$ while
 ${\Map}^{f \neq 0}\big(S^4_\QQ, S^4_\QQ\big)\simeq S^7_\QQ$. } 
We are interested in the components that contain the constant map, 
as we were in \cite{SV1}\cite{SV2}.
Modding out by the action of the torus, we need to check the nilpotence of the toroidification,
as appropriate for rational homotopy theory \cite{Bousfield-Gugenheim}\cite{Hilton82}. We do this in \Cref{Prop-nil}.


 \vspace{-3mm} 
\paragraph{Automorphisms.} 

The Lie groups and Lie algebras that we extracted in \cite{SV1}\cite{SV2} from the rational homotopy models 
$M(\mathcal{L}_c^k S^4)$ of the cyclifications were in the form 
of automorphisms and derivations. As indicated above,
we will be replacing the cyclic loop spaces with the 
toroidifications, and schematically we will have 
\vspace{-4mm} 
\(
\label{AutDer}
\hspace{9mm}
\xymatrix@R=-8pt{
& \protect\overbrace{\fcolorbox{black}{lightgray}{\text{ $M = M(\mc{T}^k S^4)$}}}^{\rm \bf \color{darkblue} Model \; of \; toroidification} \ar@{<->}[dr] \ar@{<->}[dl]& 
\\
 \protect\underbrace{\fbox{\text{Automorphisms Aut$(M)$}}}_{\rm \bf \color{darkblue} Lie \; Group} \;\; \; \; \; \; 
 \ar@{..>}[rr]_-{\rm Linearization/ differential \; of \; map}&  & \;\; 
 \protect\underbrace{\fbox{\text{Derivations Der$(M)$}}}_{\rm \bf \color{darkblue} Lie \; Algebra}
}
\)

\vspace{-6mm} 
\paragraph{Derivations.} 
In geometry, the Lie algebra of the group of diffeomorphisms of a manifold 
is the Lie algebra of vector fields, 
i.e., derivations of the algebra of functions. 
Also, in the case of a commutative algebra  $\mathfrak{c}$,
the automorphism group
of $\mathfrak{c}$,
\vspace{-1mm} 
$$
{\rm Aut}(\mathfrak{c}):=\big\{ \phi \in {\rm GL}(\mathfrak{c}) \;|\; 
\phi(g_1  g_2)= \phi(g_1) \phi(g_2), \; \phi(1) = 1 \quad
\forall g_1, g_2 \in \mathfrak{c} \big\},
$$

\vspace{-1mm} 
\noindent
is similarly related to the
Lie algebra of derivations,
\vspace{-1mm} 
$$
{\rm Der}(\mathfrak{c}):=\big\{ D \in \mathfrak{gl}(\mathfrak{c}) \;|\; 
D(g_1 g_2)= D(g_1) g_2 + g_1 D(g_2), \; D(1) = 0 \quad
\forall g_1, g_2 \in \mathfrak{c}) \big\}.
$$

\vspace{-1mm} 
\noindent
See \cite[III.10]{Bourbaki}\cite{GAS95} for the classical theory and \cite{schlessinger-stasheff2} 
for the case of graded algebras.

\smallskip 
The situation for rational models is similar but subtler. Let $X$ be a nilpotent space of finite type, $X_\Q$ be its rationalization, $M(X)$ be the Sullivan minimal model of $X$, and ${\Der}^\bullet M(X)$ be the 
differential graded Lie algebra (DGLA) of graded derivations of the underlying graded algebra of $M(X)$. 
Let $\Aut X$ be the topological monoid of homotopy self-equivalences of $X$ and ${\Aut}^h(X_\Q)$ be the (discrete) group of homotopy classes of homotopy self-equivalences of $X_\Q$. By the main theorem of rational homotopy theory, the latter group is isomorphic to the group ${\Aut}^h M(X)$ of homotopy classes of automorphisms of $M(X)$.
For $X$ a nilpotent CW complex which is either finite or has a finite Potsnikov tower (i.e., finite number 
of nonvanishing homotopy groups), the groups $\Aut M(X)$ of automorphisms of $M(X)$ and ${\Aut}^h M(X) \cong {\Aut}^h (X_\Q)$ are 
affine algebraic group schemes over $\Q$ \cite{Su77}\cite{BL05}. Also, with Sullivan's hint \cite[\S 11]{Su77} on the computation of rational homotopy groups of  $B \Aut X$ for a finite, simply connected CW complex $X$,
Tanr\'{e} \cite{tanre} and Schlessinger-Stasheff \cite{schlessinger-stasheff2} showed that the negative truncation $\Der^{\langle-1 \rangle} M(X)$ of ${\Der}^\bullet M(X)$ determines the rational homotopy type of 
the simply connected cover of the classifying space 
$B{\Aut} X$, that is,  $\Der^{\langle-1 \rangle} M(X)$ is a Quillen DGLA-model of the simply connected cover of $B{\Aut} X$.
Furthermore, $\pi_{1} B \Aut X \linebreak[0] = \Aut^h X$.
The Lie algebra of the algebraic group $\Aut M(X)$ may be identified with the Lie algebra $\Der M(X) := Z^0(\Der^\bullet M(X))$ of derivations of the DGCA $M(X)$, i.e., degree-zero derivations that commute with the differential in $M(X)$. The Lie algebra of the group $\Aut^hM(X)$ is identified in \cite{BL05} as the degree-zero cohomology $H^0(\Der^\bullet M(X))$ as well as the degree-zero Andr\'{e}-Quillen cohomology 
$H^0_{\operatorname{AQ}}(M(X), M(X))$. 
The rational homotopy properties of the DGLA of derivations are studied in \cite{tanre}\cite{LS15}. 
Extracting the derivations generally requires getting into explicit computations 
(see \cite{Gat97}\cite{schlessinger-stasheff2}), which is what we do in 
\cref{parabolic_action}.

\medskip 
Applying \cite{Su77}\cite{tanre}\cite{schlessinger-stasheff2} to $X=S^4$, we have 
$
\pi_{n}^\Q\big({\Aut} (S^4)\big) \cong H^{-n}\big({\rm Der}^{\langle-1\rangle}(M(S^4)) \big)
$ for $n \ge 1$.
In general, for a map $f: X \to Y$ with minimal Sullivan model $M(f): M(Y) \to M(X)$, with $X$ and $Y$  nilpotent
and $X$  finite, there is an isomorphism \cite{BL05}\cite{LS07}\cite{BM08}  (see \cite{Smith10} for a survey)
$
\pi_n^\Q \big({\Map}^f(X, Y)\big) \cong H^{-n}\big({\rm Der}^{\langle-1 \rangle}(M(Y), M(X); \linebreak[0] M(f)) \big)
$ for $n \ge 2$.
Taking $X=T^k$ and $Y=S^4$, we have
\vspace{-2mm} 
$$
\pi_n^\Q\big(\mathcal{L}^k S^4\big) \cong H^{-n}\big({\rm Der}^{\langle-1 \rangle}(M(S^4), M(T^k); M(0)) \big)\,.
$$



\vspace{-3mm} 
\paragraph{Degrees.}
If $A = M(X)$ is the Sullivan minimal model of a nilpotent space $X$ of finite type, then the group $\Aut A$  of automorphisms of $A$
is a rational algebraic group, which
we considered in \cite{SV1}, where we described its maximal split torus.
The Lie algebra of $\Aut A$ is the non-dg version of the Lie algebra $\Der A$ of derivations of $A$,i.e., degree-zero derivations of $A$ which commute with the differential. 
In considering the graded version,
there is more than one convention. 
We follow \cite{Good}\cite{BL05}\cite{schlessinger-stasheff2}\cite{BS23}, which use the positive-degree convention, different from \cite{Su77}\cite{tanre}\cite{LS07}\cite{Smith10}:
for us, a derivation of degree $n$ \textit{increases} the degree in $A$ by $n$. 
Thus, $\Der^\bullet A$ becomes a dg-Lie algebra with a differential of degree $+1$.
The ``derived,'' dg-version has the form $\Der^\bullet A = \bigoplus_{n=-\infty}^{\infty} \Der^n A$, with $\Der A = Z^0(\Der^\bullet A)$.

\paragraph{Symmetries.}
 The global symmetry that is generally expected to arise in the 
lower-dimensional supergravity theory upon reduction on $T^k$ comprises the Lie algebra
$\mathfrak{p}_k^{k(k)}=\mathfrak{gl}(k, \RR) \ltimes \RR^{n_k}$, where $n_k$ is a polynomial in $k$.
The first factor comes from the diffeomorphisms of the internal space
and the second factor comes from the shifts of the scalar fields
associated with the reduction and/or dualization of the form-field potentials
\cite{CJLP2}.
The Lie algebra that appears in maximal supergravity is the
larger, exceptional Lie algebra
$\mathfrak{e}_{k(k)}$, of which the above $\mathfrak{p}_k^{k(k)}$ is a maximal parabolic subalgebra. The main goal of this paper is to construct an action of $\mathfrak{p}_k^{k(k)}$ on the Sullivan minimal model $M(\mc{T}^k S^4)$ of the toroidification $\mc{T}^k S^4$, i.e., a Lie algebra homomorphism
\vspace{-3mm}
\begin{equation}
\label{parabolic-action}
\mathfrak{p}_k^{k(k)} \longrightarrow \Der M(\mc{T}^k S^4).
\end{equation}

\vspace{-2mm} 
The toroidification $\mathcal{T}^k S^4$ may be viewed as the \emph{universal target of $(11-k)$-dimensional supergravity} in rational homotopy theory via the continuous map
\vspace{-2mm}
\begin{equation}
\label{Phi}
X \longrightarrow \mc{T}^k S^4_\R
\end{equation}

\vspace{-2mm} 
\noindent adjoint to the DGCA map $M(\mc{T}^k S^4) \to \Omega^\bullet(X)$ given by mapping the generators of $M(\mathcal{T}^k S^4)$ to the corresponding supergravity form-fields on spacetime $X$, $\dim X = 11-k$, see \cite{V-RHT}. (The algebra map respects the differentials, because the differentials of the generators of $M(\mathcal{T}^k S^4)$ are the \emph{equations of motion $($EOMs$)$} of supergravity. This is the \emph{general Hypothesis H} for the reduction of M-theory on $T^k$, see \cite{SV2}.) The Lie group $\Aut M(\mathcal{T}^k S^4)$ and thereby its Lie algebra $\Der M(\mathcal{T}^k S^4)$, considered above, act on the generators of $M(\mathcal{T}^k S^4)$, respecting the differential. Then, by our main construction, the parabolic subalgebra $\mathfrak{p}_k^{k(k)}$ will also act on $M(\mathcal{T}^k S^4)$ via the homomorphism \eqref{parabolic-action}. We may view this action as the universal action on the form-fields respecting the EOMs. Note that this action does not induce an action on the actual form-fields defined in spacetime. It rather means that we have an action on certain objects (polynomial differential forms) on $\mathcal{T}^k S^4_\R$ which pull back to the form-fields along the map \eqref{Phi}. One can interpret these symmetries as \emph{infinitesimal symmetries of the equations of motion} in spacetime.

\newpage 
To move from Lie algebras to groups of symmetries, we note that the action \eqref{parabolic-action} of the parabolic subalgebra canonically induces an action of the corresponding simply connected Lie group $P_k^{k(k)}$ by automorphisms of $M(\mathcal{T}^k S^4)$:
\vspace{-2mm}
\[
P_k^{k(k)} \longrightarrow \Aut M(\mathcal{T}^k S^4).
\]

\vspace{-2mm} 
\noindent In view of the discussion of derivations of DGCAs above, we get an action of the Lie group $P_k^{k(k)}$ on $\mathcal{T}^k S^4_\R$ in the rational homotopy category over $\R$:
\vspace{-2mm}
\[
P_k^{k(k)} \longrightarrow \Aut M(\mathcal{T}^k S^4) \longrightarrow \Aut^h M(\mathcal{T}^k S^4) = \Aut^h \mathcal{T}^k S^4_\R.
\]

\vspace{-2mm}
\noindent This may be regarded as a \emph{topological realization of the symmetries of EOMs of supergravity} in $11-k$ dimensions.


\paragraph{Split real forms.} 
In \cite{SV1}\cite{SV2} we studied an action on $M(\mc{L}_c^k S^4)$ of the real split torus 
corresponding to the Cartan subalgebra of the split real form $\mathfrak{e}_{k(k)}$ of $\mathfr{e}_k$. 
Here, we also construct an action of the same torus on the RHT model $M(\mc{T}^k S^4)$ of the toroidification and
extend it to an action of the maximal parabolic subalgebra of $\mathfrak{e}_{k(k)}$. The standard split real form is obtained by taking only real combinations of the 
Chevalley generators, i.e., as the subalgebra of fixed points of the 
standard complex conjugation in that basis. 
This amounts to simply using $\mathbb{R}$ in place of $\mathbb{C}$ in the
definition of the split form.
See \cite[App. A]{HKN} and references therein. 
The split real form $\mathfrak{e}_{k(k)}$ is uniquely characterized as a real form $\g_0$ of $\mathfr{e}_k$ that has a Cartan subalgebra acting on $\g_0$ with real eigenvalues
(see, e.g., \cite[Chapter 26]{FH04}).
There is a direct sum decomposition 
$\mathfrak{e}_{k(k)}=\mathfrak{h} \oplus (\bigoplus \mathfr{e}_{k(k)}^\alpha)$ into 
the Cartan and one-dimensional weight spaces for the action of $\mathfrak{h}$; for each root $\alpha$ of $\mathfr{e}_k$, the subspace $\mathfr{e}_{k(k)}^\alpha$ is the intersection of the corresponding root space $\mathfrak{e}_{k}^\alpha
\subset \mathfrak{e}_k$ with $\mathfrak{e}_{k(k)}$. Each pair $(\mathfr{e}_{k(k)}^\alpha, \mathfr{e}_{k(k)}^{-\alpha})$
generates a sublagebra isomorphic to $\mathfrak{sl}(2, \RR)$.

\paragraph{The parabolic subalgebras.\!\!} 
Parabolic subalgebras in the complex case are discussed extensively in \cite{Ku02}. 
One can form the theory in general for any field \cite{Rous}\cite{Rous2}
(see also \cite{CW22}). The split real form is again obtained by replacing $\CC$ with $\R$ in the complex form. The real forms of the exceptional simple Lie algebras and their maximal parabolic subalgebras 
are explicitly determined in \cite{Dobrev08} (see also \linebreak[3] \cite{Dobrev13}\cite{Dobrev16}\cite{DM20}).
We describe those in \Cref{sec-para} while also providing physical interpretations, using \cite{LPS98} (see  
\cite{ADFFMT97}\cite{ADFFT97} also for a related but different perspective). 

\smallskip 
A choice of Cartan subalgebra $\mathfrak{h}$ determines a decomposition 
$\mathfrak{e}_{k(k)}=\mathfrak{h} \oplus \big(\!\bigoplus_{\alpha \in \Delta} \, \mathfrak{e}_{k(k)}^\alpha\big)$. 
To each choice of ordering of the root system 
$\Delta = \Delta^+ \cup \Delta^-$, one can associate a minimal parabolic, i.e., the \emph{Borel subalgebra}
$\mathfrak{b}= \mathfrak{h}\oplus (\bigoplus_{\alpha \in \Delta^+} \, \mathfrak{e}_{k(k)}^\alpha)$.
 We are interested in maximal (proper) parabolics, which contain 
$\mathfrak{b}$ as a subalgebra and have a corresponding decomposition 
$\mathfrak{p}(\Sigma)= \mathfrak{h}\oplus (\bigoplus_{\alpha \in \Delta(\p)} \, \mathfrak{e}_{k(k)}^\alpha)$, where for $k \ge 3$, $\Delta(\p)$ consists of all roots which can be written as sums of negatives of the roots in a fixed maximal proper subset $\Sigma$  of the set of simple roots, together with all positive roots. For $k \le 2$, the setup is somewhat different, see \Cref{small-k}.
Hence, parabolic subalgebras are in one-to-one correspondence 
with the nodes of the Dynkin diagram, taken one at a time. 
We are interested in the special node
$\alpha_k$ in the Dynkin diagram
\(
\label{Dynkin-schem}
\hspace{-1cm} 
\qquad
  \scalebox{.9}{$
  \raisebox{-30pt}
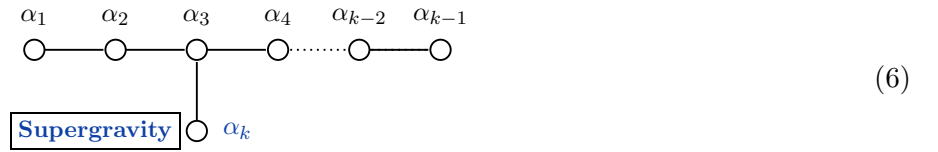
{\begin{tikzpicture}[scale=.6]
    \foreach \x in {0,...,5}
    \draw[thick,xshift=\x cm] (\x cm,0) circle (2.5 mm);
    \foreach \y in {0, 1,2, 4}
    \draw[thick,xshift=\y cm] (\y cm,0) ++(.25 cm, 0) -- +(14.5 mm,0);
    \foreach \y in {3,4}
    \draw[dotted, thick,xshift=\y cm] (\y cm,0) ++(.3 cm, 0) -- +(14 mm,0);
    \draw[thick] (4 cm, -2 cm) circle (2.5 mm);
    \draw[thick] (4 cm, -3mm) -- +(0, -1.45 cm);
    \node at (0,.8) { $\alpha_1$};
       \node at (2,.8) { $\alpha_2$};
    \node at (4,.8) {$\alpha_3$};
    \node at (6,.8) {$\alpha_4$};
    \node at (8,.8) {$\alpha_{k-2}$};
    \node at (10,.8) {$\alpha_{k-1}$};
        \node at (5,-2) {\color{darkblue} $\alpha_k$};
    \node at (1.5,-2) {\fbox{\color{darkblue} \bf \small  Supergravity}};
  \end{tikzpicture}
  }
  $}
\)
which for $k = 3$ degenerates into
\vspace{-2mm}
\(
\label{Dynkin-schem-3}
\hspace{-1cm} 
\qquad
  \scalebox{.9}{$
  \raisebox{-20pt}
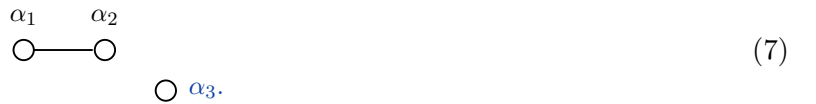
{\begin{tikzpicture}[scale=.6]
    \foreach \x in {0,1}
    \draw[thick,xshift=\x cm] (\x cm,0) circle (2.5 mm);
    \draw[thick] (0,0) ++(.25 cm, 0) -- +(14.5 mm,0);
    \draw[thick] (3.5 cm, -1 cm) circle (2.5 mm);
    \node at (0,.8) { $\alpha_1$};
       \node at (2,.8) { $\alpha_2$};
        \node at (4.5,-1) {\color{darkblue} $\alpha_3$.};
  \end{tikzpicture}
  }
  $}
\)
For $k = 2$, the Dynkin diagram just becomes
$
\hspace{-.8cm} 
\qquad
  \scalebox{.9}{$
  \raisebox{-2pt}{\begin{tikzpicture}[scale=.5]
    \draw[thick] (0,0) circle (2.5 mm);
    \node at (.05,.65) { $\alpha_1$};
  \end{tikzpicture}
  }
  $}
$\!\!, and for $k=0, 1$, it is empty. However, noting that for $k=2$, the Dynkin diagram is of type $A_1$, which corresponds to the Lie algebra $\mathfr{sl}(2)$, we find it natural to think of the case $k=1$ as corresponding to $\mathfr{sl}(1) = 0$ and the case $k=0$ as corresponding to $\mathfr{sl}(0) := \varnothing$ and denote these two cases by $A_0$ and $A_{-1}$, respectively, see \Cref{table1}.

\medskip 
The maximal parabolics can also be seen via the Weyl chamber. Associated to the node $\alpha_i$ is the edge of the Weyl chamber 
given by 
$\mathcal{W}_{\alpha_i}=\{\lambda ~|~ (\lambda, \alpha_i)>0, (\lambda, \alpha_j)=0 \; \text{for} \; j\neq i \}$,
where $(-,-)$ is the inner product on $\h^*$ induced by the Killing form. Hence there is a correspondence (see \cite{Humphreys}\cite{CollingwoodMcGovern}\cite{FH04})
\vspace{-2mm}
$$
\xymatrix{
\colorbox{lightgray}{Single nodes of Dynkin diagram}
\ar@{<->}[r] & \colorbox{lightgray}{Edges of Weyl chamber}
\ar@{<->}[r] & \colorbox{lightgray}{Maximal parabolic subalgebras}
}
$$

\vspace{-7mm}
\paragraph{The symmetry patterns of  $E_k$  and their parabolics.} 
\medskip 
Considering the toroidification models of the 4-sphere, we establish 
actions on them of the parabolic subalgebras.
The structures and patterns are captured in Table 1, see details in \cref{sec-para} and \cref{parabolic_action}. The table reveals only part of the story:
in \cref{parabolic_action} we define actions of larger Lie algebras, namely, the maximal parabolic Lie algebras $\p_k$ of certain (trivial) central extensions $\g_k$ of the Lie algebras $\mathfr{e}_{k(k)}$. Note that the case $k=9$ is special, because the Lie algebra $\mathfr{e}_{9(9)}$, being an affine Kac-Moody algebra, already incorporates a (nontrivial) central extension, and we have $\g_9 = \mathfr{e}_{9(9)}$. This is why the parabolic subalgebra in the table for $k=9$ is larger than for other values of $k$, and we see $\mathfr{gl}$ instead of $\mathfr{sl}$ here.

\vspace{-1mm} 
{\small 
\begin{table}[H]
\centering
\begin{tabular}{ccllcl}
\hline
$k$ &  $D$ & {\bf Type $E_{k}$} & {\bf Split real Lie algebra} 
&  
\!\!{\bf Min.\ model} & {\bf Max.\ parabolic} $\mathfr{p}_{k}^{k(k)}$  \\
\hline
 \hline
\rowcolor{lightgray} 0 & 11 & $E_{0}=A_{-1}$ & $\mathfrak{e}_{0(0)} =\mathfr{sl}(0, \RR) = \varnothing$ 
&  $M(S^4)$       &  $\varnothing$ 
\\
1 & 10 & $E_{1}=A_0$ & $\mathfrak{e}_{1(1)} =\mathfr{sl}(1, \RR) = 0$ 
&  $M(\mc{T}^1 S^4)$ & $0$ 
\\
\rowcolor{lightgray} 2 & 9 & $E_{2}=A_{1}$ & $\mathfrak{e}_{2(2)} =\mathfr{sl}(2, \RR)$ 
&  $M(\mc{T}^2 S^4)$ &  $
\mathfr{sl}(2, \RR)$ 
\\
3 & 8 & $E_{3}=A_2 \times A_1$ &$\mathfrak{e}_{3(3)} =\mathfr{sl}(3, \RR) \, \times \,\mathfr{sl}(2, \RR)$  
&  $M(\mc{T}^3 S^4)$ &   $ \mathfrak{sl}(3, \RR)\, \oplus \, \mathfrak{so}(1,1) \, \oplus \, \mathfrak{n}^1 \; \,$
\\
\rowcolor{lightgray} 4 & 7 & $E_{4}=A_4$& $\mathfrak{e}_{4(4)} =\mathfr{sl}(5, \RR)$
&  $M(\mc{T}^4 S^4)$ & $\mathfrak{sl}(4, \RR)\, \oplus \,
\mathfrak{so}(1,1) \, \oplus \, \mathfrak{n}^4 \; \,$
\\
5 & 6 & $E_{5}=D_5$& $\mathfrak{e}_{5(5)} =\mathfr{so}(5,5)$  
&  $M(\mc{T}^5 S^4)$ & $\mathfrak{sl}(5, \RR)\, \oplus \,
\mathfrak{so}(1,1)
\, \oplus \,\mathfrak{n}^{10}$
\\
\rowcolor{lightgray} 6 & 5 & $E_{6}$&$\mathfrak{e}_{6(6)}$  
&  $M(\mc{T}^6 S^4)$ &  $ \mathfrak{sl}(6, \RR)\, \oplus \, \mathfrak{so}(1,1)  \, \oplus \, \mathfrak{n}^{21}$
\\
7 & 4 & $E_{7}$& $\mathfrak{e}_{7(7)}$   
&  $M(\mc{T}^7 S^4)$ &  $\mathfrak{sl}(7, \RR)\, \oplus \,\mathfrak{so}(1,1) \, \oplus \,\mathfrak{n}^{42}$
\\
\rowcolor{lightgray} 8 & 3 & $E_{8}$&$\mathfrak{e}_{8(8)}$ 
&  $M(\mc{T}^8 S^4)$ &   $\mathfrak{sl}(8, \RR) \, \oplus \, \mathfrak{so}(1,1) \, \oplus \,
\mathfrak{n}^{92}$\\ 
9 & 2 & $E_{9}$&$\mathfrak{e}_{9(9)}$ 
&  $M(\mc{T}^9 S^4)$ &   $\mathfrak{gl}(9, \RR) \, \oplus \, \mathfrak{so}(1,1) \, \oplus \,
\mathfrak{n}^{\infty}$\\ 
\rowcolor{lightgray} 10 & 1 & $E_{10}$&$\mathfrak{e}_{10(10)}$ 
&  $M(\mc{T}^{10} S^4)$ &   $\mathfrak{sl}(10, \RR) \, \oplus \, \mathfrak{so}(1,1) \, \oplus \,
\mathfrak{n}^{\infty}$\\ 
11 & 0 & $E_{11}$&$\mathfrak{e}_{11(11)}$ 
&  $M(\mc{T}^{11} S^4)$ &   $\mathfrak{sl}(11, \RR) \, \oplus \, \mathfrak{so}(1,1) \, \oplus \,
\mathfrak{n}^{\infty}$\\ 
\hline
\end{tabular}
\vspace{-2mm} 
\caption{\label{table1} \footnotesize The $E_k$ pattern in Lie theory, $D$-dimensional supergravity,
 and toroidifications of $S^4$.}
\end{table}
}




\vspace{-7.5mm} 
\paragraph{Results.} In this paper, we establish the following. 
\vspace{-2mm} 
\begin{itemize} 
 \setlength\itemsep{-4pt}
\item[{\bf (i)}] In \cite{SV1}, we observed that the
result of Vigu\'{e}-Poirrier and Sullivan \cite{VPS76}
on the minimal model of the free loop space of a simply connected space
holds in the more general 
case of a nilpotent, path-connected space.
Below in \Cref{free}, we generalize this further to free $k$-fold loop spaces. 

\item[{\bf (ii)}]  We observed in \cite{SV1} that the result of Vigu\'{e}-Poirrier and Burghelea \cite{vigue-burghelea}
on the minimal model of the cyclic loop space of a simply connected space $X$ holds for $X$ being nilpotent and path-connected
(Theorem \ref{CyclicModel} below). 
In Theorem \ref{tor}, we generalize this further to toroidification, after we 
demonstrate nilpotency in Prop. \ref{Prop-nil}.

\item[{\bf (iii)}] For the algebraic models, we establish an algebraic toroidification 
result in Theorem  \ref{dgca-adj} (and a truncated version in Corollary \ref{dgca-adj-trunc}), 
analogous to the topological adjunction 
of \cite{BSS} between the totalization and equivariant mapping space functors.
Shifting from the space to the algebra level allows for computability. 

\item[{\bf (iv)}] 
In Theorem \ref{p-action}, we establish a linear action of the parabolic Lie subalgebra 
$\p_k$ of the algebra $\g_k$ of type $E_{k(k)}$ on the Sullivan minimal toroidification model $M(\mc{T}^k S^4)$. This is explicit via the Chevalley generators of $\g_k$ (viewed as a Kac-Moody algebra) in both the finite-dimensional ($0\leq k \leq8$)
and the infinite-dimensional cases ($k \geq 9$). 

 \item[{\bf (v)}] For small $k$, we extend this in Theorem \ref{g-action} to a full action of the exceptional 
 Lie algebra $\g_k$.

\end{itemize} 

\vspace{-4mm} 
\paragraph{Physics consequence.} 
As a consequence of results in items {\bf (iv)} and {\bf (v)} above,  
we have a long-sought-after action of (maximal parabolic subalgebras of) the exceptional Lie algebras on the corresponding spaces of 
fields and their dynamics, captured universally by the toroidification spaces, 
both for the finite-dimensional Lie algebras and the infinite-dimensional (affine, hyperbolic, and more generally, Lorentzian) Kac-Moody algebras. 
This is a major step in establishing U-duality covariance in this setting. 

\paragraph{Acknowledgments.}
We are grateful to Ben Brubaker, Kentaro Hori, Urs Schreiber, and Taizan Watari for helpful discussions. 
The first author acknowledges the support by Tamkeen under the NYU Abu Dhabi Research Institute grant CG008.
The
work
of the second author
was 
supported by the World Premier International Research Center Initiative (WPI Initiative), MEXT, Japan, 
and a Collaboration grant from the Simons Foundation (\#585720).

\section{The Sullivan minimal model of the toroidification}
\label{Sec-toroidMod}

We will work over the ground field $\R$ throughout the paper, even in rational homotopy theory, unless stated otherwise. One reason is Hypothesis H, see \eqref{Phi}, uses the de Rham algebra of the spacetime and is inherently real. The other is that we want to focus on real Lie groups and real forms of Lie algebras, common to U-duality.

We will  consider, as in the Introduction, the $k$-fold free loop space $\mc{L}^k S^4 := \Map^0(T^k, S^4)$
and \emph{toroidification}
$\mc{T}^k S^4 := \mc{L}^k S^4\dslash T^k$. The notation $\Map^0$ indicates the path
component of the constant map, as the full mapping space $\Map$ would not be
connected for $k \ge 4$.
First, we set up some notation. Let $V = \bigoplus_{n} V^n$ be a graded vector
space. Define its \emph{connective looping $($truncated desuspension$)$} as the graded vector space $V[1]
= \bigoplus_{n} V[1]^n$ with
\begin{equation}
\label{cl}
V[1]^n := 
\begin{cases}
    V^{n+1}, & \text{if $n > 0$,}\\
    0 & \text{otherwise.}
\end{cases}
\end{equation}
In the ``simply connected'' case, i.e., 
when $V^n = 0$ for $n \le 1$, $V[1]$ just becomes the usual looping (desuspension) $V[1]^n = V^{n+1}$ for all $n \in \ZZ$.
Let
\vspace{-3mm} 
\begin{align*}
  s: V & \longrightarrow V[1],
  \\[-2pt] 
  v & \longmapsto sv \coeq v,
\end{align*}
be the natural map of degree $-1$. In \cite{SV1}, we observed that the following 
result of Vigu\'{e}-Poirrier and Sullivan \cite{VPS76} for a simply connected space of finite type
holds in the more general 
case of a nilpotent, path-connected space $X$ of finite type, the only difference being the use of connective looping above.

\begin{theorem}[Minimal model of free loop space]
\label{free-loop}
Let $X$ be a nilpotent, path-connected space of finite type
with Sullivan minimal model $(S(V), d)$. 
  Then the Sullivan minimal model of the free loop space
  $\mc{L} X = \Map^0(S^1, X) = \Map (S^1, X)$ is isomorphic to
  \vspace{-1mm} 
  $$
  (S(V \oplus V[1]), d_f)\,,
  $$
  with the differential $d_f$ defined
  by
\vspace{-3mm} 
\begin{align*}
 d_f v  := dv, \qquad 
 d_f sv  := -s dv,
\end{align*}
 where $s$ is extended to a unique degree-$(-1)$ derivation
of $S(V \oplus V[1])$ such that $s(u) = 0$ for $u \in V[1]$ (which implies $s^2 = 0$).
  \end{theorem}

It is useful to think of the desuspension map $s$ as
an operation on $S(V \oplus V[1])$, especially when we deal with the
following iteration of the above construction.

\begin{prop}[Minimal model of free $k$-fold loop space]
  \label{free}
  Let $X$ is a nilpotent,
  path-connected space of finite type with Sullivan minimal model $(S(V), \linebreak[0] d)$.
  Them the Sullivan minimal model of the $k$-fold free loop space $\mc{L}^k
  X := \Map^0 (T^k, \linebreak[0] X)$  is
  freely generated as a graded commutative, unital algebra by a
  graded vector space obtained as follows. 

\vspace{-3mm} 
 \begin{itemize}[leftmargin=22pt]
 \setlength\itemsep{-2pt}
 \item[{\bf (a)}]   Take
  the graded commutative algebra $\RR[s_1, \dots , s_k]$ on $k$
  degree-$(-1)$ variables and generate a free graded $\RR[s_1, \dots ,
    s_k]$-module
  \[
U := \RR[s_1, \dots , s_k] \otimes V.
\]
\item[{\bf (b)}] Then take the positive-degree part $U^{>0}$ of the result and generate
a free graded commutative algebra, $S(U^{>0})$. This is the underlying
graded commutative algebra of the Sullivan minimal model $M(\mc{L}^k
X)$ of $\mc{L}^k X$.

\item[{\bf (c)}] The differential $d_f$ is determined by the relations
\vspace{-2mm} 
  \begin{align*}
    d_f v :=& \, dv, \qquad \text{for $v \in V$},\\
    [d_f,s_i]  =& \, 0,    \qquad \;\; 1 \le i \le k,
  \end{align*}

  \vspace{-3mm} 
\noindent
  where the operator of multiplication
  \vspace{-3mm} 
  \begin{align}
  s_i: U^{> 0} &\longrightarrow U^{>0},
  \label{s-op}
  \\[-2pt]
  v & \longmapsto s_i v,
  \nonumber 
\end{align}

\vspace{-3mm} 
\noindent is extended to $S(U^{>0})$ as a degree-$(-1)$ derivation.
\end{itemize} 
  \end{prop}

  \begin{example}[$k$-fold free loop space of the 4-sphere]
  More explicitly, for $X = S^4$ we have:
  \vspace{-1mm} 
  \begin{align*}
    M(\mc{L}^k S^4) & = \left(\RR[s_{i_1} \dots s_{i_p} g_4, s_{j_1}
      \dots s_{j_q} g_7], \;
      d_f\right), \;\;
  \scalebox{0.8}{$ 0 \le p \le 3, \;\; 1 \le i_1 < \dots < i_p \le k, \;\;
    0 \le q \le 6, \;\; 1 \le j_1 < \dots < j_q \le k,$}
    \\
    d_fg_4 &= 0, \qquad \qquad d_fg_7 = -\tfrac{1}{2} g_4^2,\\
    d_f s_i x &= -s_i d_f x \qquad \text{for all $x \in M(\mc{L}^k S^4)$} \;
    .
  \end{align*}
\end{example} 
There is also a computation of the Sullivan minimal model of the 
cyclic loop space (or cyclification) $\mc{L}_c X$, which is same as the single toroidication $\mc{T}^1 X$ of $X$. We observe in \cite{SV1} that the following
result of Vigu\'{e}-Poirrier and Burghelea \cite{vigue-burghelea}, originally done for a simply connected space,
holds for $X$ being nilpotent and path-connected, if $V[1]$ is understood as the connective looping of $V$ (see \eqref{cl}).

\begin{theorem}[Minimal model of cyclic loop spaces]
  \label{CyclicModel}
  Let $X$ be a nilpotent, path-connected space of finite type with Sullivan minimal model
  $(S(V), d)$. Then the Sullivan minimal model of the cyclification
  $\mc{L}_c X = \mc{T}^1 X = \Map^0(S^1, X) \dslash S^1 = \Map(S^1, X) \dslash S^1$ is isomorphic to
  \vspace{-3mm} 
  \[
  (S(V \oplus V[1]) \otimes \RR[w]), d_c),
  \]

  \vspace{-1mm} 
\noindent with the differential $d_c$ defined on all $v \in V$ by
\vspace{-2mm} 
\begin{align*}
d_c v := d v + w \cdot sv\, ,\qquad 
  d_c sv  := -s d v\, , \qquad 
  d_c w  := 0\,,
\end{align*}

\vspace{-2mm} 
\noindent  where $s$ is extended to a unique degree-$(-1)$ derivation
of $S(V \oplus V[1])$ such that $s(u) = 0$ for $u \in V[1]$, as in \Cref{free-loop}.
  \end{theorem}

Of course, one can identify the Sullivan minimal model of the iterated cyclification $\mc{L}_c^k X$, 
just as in \Cref{free}, but this space is quite different from the toroidification $\mc{T}^k X$ and, 
as far as we know, the Sullivan minimal model of it has not been identified. To that end, we first
establish nilpotency in an analogous way to the one we used to handle the cyclification in \cite{SV1}.

\begin{prop}[Basic properties of  the toroidification space] 
\label{Prop-nil}
If $X$ is path-connected, nilpotent, and  of finite type,
then so is its $k$-toroidification $\mc{T}^k X := \mc{L}^k X \dslash T^k$. 
  \end{prop}

\begin{proof}
The case $k=1$ is done in  \cite[Prop. 2.4]{SV1}, and extending this to general $k$ goes 
through similarly. First of all, it is proven there that for a path-connected, nilpotent space $X$ of finite type, the free loop space $\mc{L} X$ is of the same kind. By induction, so will be $\mc{L}^k X$.

Now, the homotopy groups of the toroidification space $\mc{T}^k X$ are the same as those of $\mc{L}^k X$, except for $\pi_{2}(\mc{T}^k X)$, which adds to $\pi_{2}(\mc{L}^k X)$ a copy of the abelian group $\ZZ^k=\pi_1(T^k)$,  
on which the conjugation action of $\pi_1(\mc{T}^k X) = \pi_1(\mc{L}^k X)$ is trivial. This all can be seen 
from the fiber sequence
\vspace{-2mm} 
\begin{equation*}
T^k \longrightarrow \mc{L}^k X \times ET^k \longrightarrow \mc{T}^k X\,,
\end{equation*}

\vspace{-2mm} 
\noindent in which the inclusion of the fiber $T^k$ into the $T^k$-orbit of the constant loop in $\mc{L}^k X$ factors 
through the contractible space $ET^k$.
\end{proof}

\begin{theorem}[Minimal model of toroidification]
\label{tor}
Let $X$  be a path-connected, nilpotent space of finite type with Sullivan minimal model
  $(S(V), d)$.
  Then the Sullivan minimal model of the toroidification $\mc{T}^k X :=
  \mc{L}^k X \dslash T^k$ 
  is freely generated as a graded
  commutative algebra by the graded vector space $U_k^{>0} \oplus \RR w_1 \oplus \dots \oplus \RR w_k$, that is to say,
  \vspace{-2mm} 
  \begin{gather*}
M(\mc{T}^k X) = S\big(U^{>0} \oplus \RR w_1 \oplus \dots \oplus \RR w_k\big) = S(U^{>0}) \otimes \RR[w_1, \dots, w_k],
\end{gather*}

\vspace{-2mm} 
\noindent   where,   as in Proposition $\ref{free}$, 
\vspace{-1mm} 
\item {\bf (i)}   $
 U_k^{>0} = \big( \RR[s_1, \dots , s_k] \otimes V \big)^{>0}
  $;
  
\item {\bf (ii)} $w_1, \dots, w_k$ are generators of degree
  $\abs{w_i} = 2$;
  
 \item {\bf (iii)} The differential $d_t$ is determined by the relations
 \vspace{-2mm} 
  \begin{align*}
    d_t v & = dv + \sum_{i=1}^k w_i \cdot s_i v, \hspace{-4cm} && \text{for $v \in V$},
    \\[-2pt]
   [d_t, s_i] & = 0, \quad d_t w_i = 0, 
    \hspace{-4cm} && \scalebox{0.8}{$ 1 \le i \le k$},
  \end{align*}

  \vspace{-1mm} 
\noindent  where the operator of multiplication $s_i$ \eqref{s-op} 
is extended to $S(U_k^{>0}) \otimes \RR[w_1, \dots, w_n]$ as a degree-$(-1)$ derivation satisfying
\vspace{-2mm} 
\[
s_i w_j = 0 \qquad \text{for all $j$}.
\]

\end{theorem}

\begin{defn}[Algebraic toroidification]
\label{trd}
We call the construction
\vspace{-2mm} 
\[
M = (S(V),d) \;\; \rightsquigarrow  \;\; \Trd_k(M) := 
\left(S \left( ( \RR[s_1, \dots , s_k] \otimes V )^{>0}\right) \otimes \RR[w_1, \dots , w_k], d_t\right)
\]

\vspace{-2mm} 
\noindent of Theorem \ref{tor} the \emph{algebraic toroidification}.
\end{defn}

\begin{example}[Reduction to eight dimensions] 
Here are explicit formulas for $X = S^4$ and $k = 3$, where we use $d$ instead of $d_t$ for the differential:
  \begin{align*}
  M(\mc{T}^3 S^4) &= \big(\RR[g_4, g_7, s_i g_4, s_i g_7, s_i s_j g_4,
    s_i s_j g_7, s_i s_j s_k g_4, s_i s_j s_k g_7, w_i], d\big), \quad  
    \scalebox{0.8}{$1 \le i, j, k \le 3$}, 
  \quad  s_j s_i = - s_i s_j, 
  \\ 
  dg_4 &=\sum_{i=1}^3 s_i g_4 \cdot w_i, \qquad \qquad \qquad \qquad dg_7 = -\tfrac{1}{2} g_4^2 +
  \sum_{i=1}^3 s_i g_7 \cdot w_i,
     \qquad   d w_i  = 0,     \qquad 
     \scalebox{0.8}{$1 \le i \le 3$},
  \\ 
  d s_i g_4 &= \sum_{j=1}^3 s_j s_i
  g_4 \cdot w_j,
  \qquad \qquad \qquad \quad ds_i g_7 = s_i g_4 \cdot g_4 + \sum_{j=1}^3 s_j
  s_i g_7 \cdot w_j, 
  \\
  d s_i s_j g_4 &=  s_i s_j s_k g_4 \cdot w_k, \quad  \scalebox{0.8}{$k \ne i, j$}, \qquad \;\; d s_i
  s_j g_7 = -s_i s_j g_4 \cdot g_4 - s_i g_4 \cdot s_j g_4 + s_i s_j
  s_k g_7 \cdot w_k, 
  \quad   \scalebox{0.8}{$k \ne i, j$}
  \\ d s_1 s_2 s_3 g_4 &= 0, \qquad \qquad \qquad \qquad \qquad  d s_1
  s_2 s_3 g_7 = 
  s_1 s_2 s_3 g_4
  \cdot g_4 + \sum_{\substack{i < j\\k \ne i,j}} \sgn \bigl( \begin{smallmatrix}1 & 2 &
    3\\ i& j & k\end{smallmatrix}\bigr)\, s_i s_j g_4 \cdot s_k g_4\,.
  \end{align*}
\end{example}

The proof of \Cref{tor} rests on an algebraic adjunction result, which mimics the topological adjunction
\vspace{-2mm} 
\begin{equation}
\label{adj}
\begin{tikzcd}
\Tot_G \co \operatorname{Ho(Spaces}_{/BG}) \arrow[r,shift left=.5ex]
&
\operatorname{Ho(Spaces)} \noloc \operatorname{Cyc}_G \arrow[l,shift left=.5ex]
\end{tikzcd}
\end{equation}

\vspace{-2mm} 
\noindent of \cite[Thm. 2.44]{BSS} between the homotopy categories of spaces and spaces over the classifying space 
$BG$ for a topological group $G$ modeled on a CW complex. Here the adjoint functors are $\Tot_G (Y) \coeq Y \times_{BG} EG$ and $\operatorname{Cyc}_G (X) \coeq \Map(G, X) \dslash G$. In other words, we have a natural bijection on homotopy classes of maps:
\vspace{-1mm} 
\begin{equation}
\label{adj1}
[Y \times_{BG} EG, X ] \xrightarrow{\quad \sim \quad} [Y, \Map(G, X) \dslash G]_{/BG} \, .
\end{equation}
Since we are interested in the rational homotopy theory of connected spaces, we will need a version of the above adjunction for the path component $\Map^0 (G,X)$ of a constant map $G \to X$, where $X$ path-connected. In other words, $\Map^0 (G,X)$ is the subspace of null-homotopic maps. Let also $[Y \times_{BG} EG, X ]_0$ denote the set of homotopy classes of maps $Y \times_{BG} EG \to X$ whose restriction to the fiber $G$ of $Y \times_{BG} EG \to Y$ is null-homotopic. Then \eqref{adj1} yields a natural bijection
\begin{equation}
\label{adj0}
[Y \times_{BG} EG, X ]_0 \xrightarrow{\quad \sim \quad} [Y, \Map^0(G, X) \dslash G]_{/BG}.
\end{equation}
This bijection localizes further, rationally to a bijection for homotopy classes of maps between rationalizations with rational null-homotopy condition on $G$ on both sides.

We also need to recall standard notions of rational homotopy theory related to fibrations.

\begin{defn}[Relative Sullivan algebras]
\label{relative}
$\,$
\vspace{-2mm} 
\begin{itemize}
\setlength\itemsep{-2pt}
\item[{\bf (i)}]
A \emph{relative Sullivan algebra} is a homomorphism of DGCAs of the form
\vspace{-2mm} 
\begin{align*}
A & \longhookrightarrow A \otimes_\R S(Z)\,,
\\[-2pt]
a & \longmapsto a \otimes 1,
\end{align*}

\vspace{-3mm} 
\noindent where $A = \bigoplus_{n \ge 0} A^n$ is a nonnegatively graded DGCA, $Z  = \bigoplus_{n > 0} Z^n $ is a positively graded real vector space and the differential 
$d$ on $A \otimes  S(Z)$ satisfies the following nilpotence condition, known as the
\emph{Sullivan condition}: $Z$ is the union of an increasing series
  of graded subspaces
  \vspace{-3mm}
\begin{equation*}
  Z(0) \subset Z(1) \subset \dots
\end{equation*}

\vspace{-2mm} 
\noindent such that $d(Z(0)) = 0$ and $d(Z(k)) \subset A \otimes S(Z(k-1))$ for $k \ge 1$.
\item[{\bf (ii)}]
We say that a relative Sullivan algebra is \emph{minimal} if
\vspace{-2mm} 
  \[
  d(Z) \; \subset \; A^+ \otimes S(Z) + A \otimes S^{\ge 2}  (Z),
  \]

  \vspace{-3mm} 
\noindent   where $A^+ \coeq \bigoplus_{n>0} A^n$.
\end{itemize}
\end{defn}

\begin{example}[Relative models]
A minimal relative Sullivan DGCA over $A = \R$
is the same as a \emph{minimal Sullivan 
DGCA}.
The inclusion $\RR[w_1, \dots, w_k] \hookrightarrow \Trd_k(M)$ is an example of a minimal relative Sullivan algebra, provided $M$ is a minimal Sullivan algebra.
If $G$ is a compact connected Lie group,
then the minimal relative Sullivan algebra $(S(W), 0) \to (S(W \oplus sW), d)$ with $W$ generated in positive even degrees, $sW := W[1]$, i.e., $\abs{sw} = \abs{w}-1$, 
and $d(sw) = w$ for $w \in W$ is a relative minimal model of the universal $G$-bundle $EG \to BG$. When $G = T^k$ is the compact $k$-torus, 
$W = \R w_1 \oplus \dots \oplus \R w_k$, $\abs{w_i} = 2$ for $1 \le i \le k$. As a  relative minimal model of an arbitrary principal $G$-bundle 
$X \to B$ associated with a map $f: B \to BG$, one may take the minimal relative Sullivan algebra $N \to N \otimes S(sW)$, where $N$ is a 
minimal model of $B$ and $d(sw) = f^*(w)$ for a minimal model $f^*: S(W) \to N$ of $f: B \to BG$. This minimal relative model is not minimal, 
in general. See details in \cite[\S 2.5.2]{FOT08}. 
\end{example}

\begin{remark}[Minimal vs.\ nonminimal models]
To every rational homotopy type of nilpotent spaces of finite type, one can functorially and 
uniquely (up to isomorphism) associate a minimal Sullivan DGCA, called the Sullivan minimal model of the space. 
A Sullivan algebra is a more general and weaker notion where the differential 
is allowed to have linear terms. 
Since we are dealing with fibrations, it is not unexpected that we would
be
using Sullivan algebras instead of minimal models, as the process of 
starting with a minimal model of the base will generally only lead to a Sullivan algebra for
the total space. See \cite{FOT08}\cite{FH17}. 
However, starting from a Sullivan algebra, one can
construct its corresponding minimal model.
In fact, one can do this systematically \cite{Bousfield-Gugenheim} (see \cite{GGMM} for an algorithm). 
\end{remark}

We will work with the following categories of DGCAs: 
\vspace{-2mm} 
\begin{itemize}[leftmargin=.6cm] 
\setlength\itemsep{-2pt}
\item  $\DGCA^{\sfree}$: DGCAs which are free as graded commutative algebras, also known as \emph{semifree} DGCAs;
\item  $\DGCA_A^{\sfree}$: DGCAs which are free as graded commutative algebras over a given semifree DGCA $A$.
\end{itemize}

\begin{defn}[Algebraic totalization]
\label{tot}
Let $W$ be a positively graded vector space. In the notation above, the functor
\vspace{-4mm} 
\begin{align*}
    \Tot_W \co \DGCA_{S(W)}^{\sfree} & \longrightarrow \DGCA^{\sfree}
    \\[-2pt]
    N  \quad &\longmapsto N \otimes S(sW),
\end{align*}

\vspace{-2mm} 
\noindent where $d(sw) = w$ for $w \in W \hookrightarrow N$, is called \emph{algebraic totalization}.

\end{defn}

Note that the DGCA-homomorphism $N \to \Tot_W (N)$ 
is a minimal relative Sullivan algebra. The functor $\Tot_W$ is an algebraic counterpart of the functor $(f: B \to BG) \mapsto f^*(EG)$ which assigns the total space of the pullback of the
universal $G$-bundle to a space $B$ over $BG$. The total space $f^*(EG)$ may also be identified with the homotopy fiber of the map $f: B \to BG$.

\smallskip 
First, let us look at an algebraic analogue of the adjunction \eqref{adj}, which we will discuss over $\QQ$. This will imply the same adjunction 
over $\RR$. It deals with a simpler version of algebraic toroidification which skips the truncation:
\vspace{-2mm} 
\vspace{-2mm} 
\begin{equation}
\label{tildeTrd}
\begin{tikzcd}
M = (S(V),d)  \; \arrow[rightsquigarrow]{r} & \; \widetilde{\Trd}_k(M) \coeq
\left(S \left(  \QQ[s_1, \dots , s_k] \otimes V \right) \otimes \QQ[w_1, \dots , w_k], d_t\right),
\end{tikzcd}
\end{equation}

\vspace{-2mm} 
\noindent 
where $d_t$ is defined as in \Cref{tor}.


\begin{theorem}[Algebraic toroidification/totalization adjunction]
\label{dgca-adj}
There is an adjunction
\vspace{-3mm} 
\[
\begin{tikzcd}
\widetilde{\Trd}_k \co \DGCA^{\sfree} \arrow[r,shift left=.5ex]
&
\DGCA^{\sfree}_{S(W)} \noloc \Tot_k \arrow[l,shift left=.5ex]
\end{tikzcd}
\]

\vspace{-3mm}
\noindent 
between the algebraic toroidification functor $\widetilde{\Trd}_k$ of $\eqref{tildeTrd}$ and the algebraic totalization 
functor $\Tot_k := \Tot_W$ of Definition $\ref{tot}$ for $W = \QQ w_1 \oplus \dots \oplus \QQ w_k$, $\abs{w_i} = 2$. 
In other words, there is a natural bijection
\vspace{-2mm} 
    \begin{equation}
    \label{bijection1}
\Hom_{\DGCA^{\sfree}} \big(M, \Tot_k( N) \big) \xrightarrow{\;\; \sim
  \;\;} \Hom_{\DGCA^{\sfree}_{S(W)}} \big( \widetilde{\Trd}_k(M), N\big),
\end{equation}

\vspace{-2mm} 
\noindent where $M \in \DGCA^{\sfree}$ and $N \in \DGCA^{\sfree}_{S(W)}$. This bijection induces a bijection on the homotopy classes of maps.
\end{theorem}

\begin{proof}

Let us construct the required bijection. Consider the algebra $\QQ[s_1, \dots, s_k]$ as in \Cref{trd} and define a pairing
\vspace{-2mm} 
    \[
    (\QQ s_1 \oplus \dots \oplus \QQ s_k) \otimes (\QQ s w_1 \oplus \dots \oplus \QQ s w_k) \longrightarrow \QQ
    \]
by $(s_i, sw_j) = \delta_{ij}$. This naturally extends to a perfect pairing
\vspace{-2mm} 
\[
(-,-):    \QQ[s_1, \dots, s_k] \otimes \QQ[s w_1, \dots, s w_k] \longrightarrow \QQ.
\]

\vspace{-2mm} 
\noindent
We can then identify the totalization with an internal Hom in the category of graded vector spaces:
\vspace{-2mm} 
\[
\Tot_k( N) = N \otimes_\QQ \QQ[s w_1, \dots, s w_k] = \iHom_\QQ  \big(\QQ[s_1, \dots, s_k], N\big),
\]

\vspace{-2mm} 
\noindent 
where the graded algebra structure on $\iHom_\QQ  (\QQ[s_1, \dots, s_k], N)$ is given by the standard (shuffle) coalgebra structure
on $\QQ[s_1, \dots, s_k]$ and convolution product on $\iHom$, while the differential is given on $h \in \iHom (\QQ[s_1, \dots, s_k], N)$ by
\vspace{-3mm} 
\begin{equation}
\label{dh}
(d h) \big(p(s_1, \dots, s_k)\big) = d h \big(p(s_1, \dots, s_k)\big) - (-1)^{\abs{h} + \abs{p}}
\sum_{i=1}^k w_i \cdot h \big( p(s_1, \dots, s_k)s_i \big).
\end{equation}

\vspace{-2mm} 
\noindent
Now, given a DGCA-homomorphism 
$f \co M \to \Tot_k( N)$, restrict it to the space $V$ of generators of $M = S(V)$ to get a linear map
\vspace{-3mm} 
\[
f = f\vert_V: V \longrightarrow \iHom  (\QQ[s_1, \dots, s_k], N)
\]

\newpage 
\noindent and use the tensor product-internal Hom adjunction to get a homogeneous linear map
\vspace{-2mm} 
\setlength{\jot}{1pt}
\begin{align*}
    F \co \QQ[s_1, \dots, s_k] \otimes V   & \;\; \longrightarrow \quad N
\\
p(s_1, \dots, s_k) v & \;\; \longmapsto
(-1)^{\abs{p} \cdot \abs{v}}
f (v) (p(s_1, \dots, s_k)),
\end{align*}
which we extend to a unique graded $S(W)$-algebra homomorphism
\vspace{-2mm} 
\[
F \co \widetilde{\Trd}_k(M) = S(\QQ[s_1, \dots, s_k] \otimes V) \otimes S(W)  \longrightarrow  N.
\]

\vspace{-2mm} 
\noindent
To see that this homomorphism respects the differentials, first note that $F\vert_M = f$. Therefore,
\vspace{-2mm} 
\[
F (dv) = f(dv) = (df) (v).
\]

\vspace{-2mm} 
\noindent
Also, recall that in $\Trd_k(M)$ we have $d_t w_i = 0$ for all $i$ and
\vspace{-4mm} 
\[
d_t p(s_1, \dots, s_k) v = (-1)^{\abs{p}} p(s_1, \dots, s_k) \bigg( dv + \sum_{i=1}^k w_i \cdot s_i v \bigg)
\]

\vspace{-2mm} 
\noindent
for all $p \in \QQ[s_1, \dots, s_k]$ and $v \in V$.
 
Now, we check
\vspace{-4mm} 
\setlength{\jot}{1pt}
\begin{align*}
F ( d_t p(s_1, \dots, s_k) v) & = (-1)^{\abs{p}} F (p(s_1, \dots, s_k) d v) + (-1)^{\abs{p}} \sum_{i=1}^k w_i \cdot F (p(s_1, \dots, s_k) s_i v) \\
& = (-1)^{ \abs{p} \cdot \abs{v}} f(dv) (p(s_1, \dots, s_k)) - (-1)^{ (\abs{p}+1)(\abs{v}+1)} \sum_{i=1}^k w_i \cdot  f(v) (p(s_1, \dots, s_k) s_i )\\
& = (-1)^{ \abs{p} \cdot \abs{v}}  (d f)(v) (p(s_1, \dots, s_k)) - (-1)^{ (\abs{p}+1)(\abs{v}+1)} \sum_{i=1}^k w_i \cdot  f(v) (p(s_1, \dots, s_k) s_i)\\
& = (-1)^{ \abs{p} \cdot \abs{v}} d \big(f(v) (p(s_1, \dots, s_k))\big) = d \big(F (p(s_1, \dots, s_k) v)\big),
\end{align*}
where we used \eqref{dh} to move to the last line.


Let us construct the inverse of \eqref{bijection1}. Given a homomorphism $F \co \widetilde{\Trd}_k(M) \to N$ of DGCAs over $S(W)$, restrict it to 
a linear map
\[
\QQ[s_1, \dots, s_k] \otimes  V \longrightarrow N
\]
on the generating space and use the tensor product-internal Hom adjunction to get a homogeneous linear map
\vspace{-2mm} 
\[
V \longrightarrow \iHom  (\QQ[s_1, \dots, s_k], N) = \Tot_k (N),
\]

\vspace{-2mm} 
\noindent 
which extends to a unique graded-algebra homomorphism
\vspace{-2mm} 
\[
f \co M = S(V) \longrightarrow \Tot_k (N).
\]

\vspace{-2mm} 
\noindent 
Let us check that it is a DGCA homomorphism:
\setlength{\jot}{-1pt}
\begin{align*}
f(dv) (p(s_1, \dots, s_k)) & = (-1)^{\abs{p}(\abs{v}+1)} F \big(p(s_1, \dots, s_k) dv \big)\\
& = (-1)^{\abs{p}(\abs{v}+1)} F \left(p(s_1, \dots, s_k) d_t v \right)  - (-1)^{\abs{p}(\abs{v}+1)} \sum_{i=1}^k F \big( w_i \cdot p(s_1, \dots, s_k) s_i v \big)\\
& = (-1)^{\abs{p} \cdot \abs{v}} d F \left(p(s_1, \dots, s_k) v \right) 
- (-1)^{\abs{p} (\abs{v}+1)} \sum_{i=1}^k w_i F\big( p(s_1, \dots, s_k)s_i v\big)\\
& = d \big(f (v) (p(s_1, \dots, s_k))\big) - (-1)^{\abs{p} + \abs{v}} \sum_{i=1}^k w_i f(v) ( p(s_1, \dots, s_k)s_i) \\
& =  (df(v)) \big(p(s_1, \dots, s_k)\big),
\end{align*}

\vspace{-2mm} 
\noindent 
where we used \eqref{dh} in the last step.

It is straightforward to see that $F \mapsto f$ is an inverse of $f \mapsto F$. The naturality also directly follows from the naturality of the constructions involved.

The statement about homotopy classes of maps comes from the observation that we can use the same constructions to go between homotopies 
$h: M \to \Tot_k(N) \otimes \Omega(1)$ and  homotopies $H: \widetilde{\Trd}_k(M) \to N \otimes \Omega (1)$.
\end{proof}

This theorem implies the following algebraic analogue of the ``connective adjunction'' \eqref{adj0}.

\begin{cor}
[Algebraic toroidification/totalization adjunction: truncated version]
\label{dgca-adj-trunc}
$\,$


\noindent {\bf (i)} There is a natural bijection
\vspace{-2mm} 
    \begin{equation}
    \label{bijection2}
\Hom^0_{\DGCA^{\sfree}} \big(M, \Tot_k( N) \big) \xrightarrow{\;\; \sim
  \;\;} \Hom_{\DGCA_{S(W)}^{\sfree}} \big( \Trd_k(M), N\big),
\end{equation}

\vspace{-2mm} 
\noindent where $M \in \DGCA^{\sfree}$ and $N \in \DGCA_{S(W)}^{\sfree}$ and 
\vspace{-2mm} 
\[
\Hom^0_{\DGCA^{\sfree}} \big(M, \Tot_k( N) \big) \coeq \big\{f \in \Hom_{\DGCA^{\sfree}} \big(M, \Tot_k( N) \big) \mid  f(M^+) \subset N^+ \otimes S(sW) \big\} .
\]

\vspace{-2mm} 
\noindent
 {\bf (ii)} This bijection induces a bijection on the homotopy classes of maps, where a homotopy on $\Hom^0$ is a DGCA homomorpshim 
$h \co M \to \Tot_k(N) \otimes \Omega(1)$ such that
\vspace{-2mm} 
\begin{equation}
\label{homotopy}
h(M^+) \subset (N \otimes \Omega(1))^+  \otimes S(sW).
\end{equation}
\end{cor}

\begin{proof}
In the construction of \Cref{dgca-adj}, the truncation $(\QQ[s_1, \dots, s_k] \otimes  V)^{>0} $ in $\Trd_k(M)$ for $M= S(V)$ on the right-hand side of \eqref{bijection2}, see \Cref{trd}, corresponds exactly to the truncation in the definition of $\Hom^0$ in \Cref{dgca-adj-trunc}. The same happens at the level of homotopies, as one can just apply the bijection \eqref{bijection2} to $N \otimes \Omega(1)$, in lieu of $N$.
\end{proof}

\smallskip 
\begin{proof}[Proof of \Cref{tor}]
First, note that if a DGCA $M$ is minimal Sullivan, then so is $\Trd_k(M)$.

Even though \Cref{dgca-adj-trunc} is not quite an adjunction, with $\Hom^0$ not being the $\Hom$ set of a category, 
there is still a unit of ``adjunction''
\vspace{-5mm}
\[
\id_{\DGCA^{\sfree}} \longrightarrow \Tot_k \circ \Trd_k\,,
\]

\vspace{-2mm} 
\noindent which implies the uniqueness of the toroidification functor $\Trd_k$ up to natural isomorphism. 
Thus, the algebra $\Trd_k(M)$ is determined up to isomorphism, unique up to homotopy satisfying the condition \eqref{homotopy}, 
which must be an actual isomorphism for $\Trd_k(M)$ being a minimal Sullivan algebra.

On the other hand, rational homotopy theory produces a (contravariant) equivalence of the (rational) homotopy
categories of spaces and DGCAs 
and the functor $Y \mapsto Y \times_{BT^k} ET^k$ corresponds to the functor $N \mapsto \Tot_k (N)$ at the 
level of DGCA's, see for example \cite[Ex. 2.72]{FOT08}. Moreover, the bifunctor $[Y \times_{BG} EG, X ]_0$ corresponds to the bifunctor $\Hom^0_{\DGCA^{\sfree}} \big(M, \Tot_k( N) \big)$. Indeed, the composition $G \to Y \times_{BG} EG \xrightarrow{\varphi} X$ is rationally null-homotopic iff the composition $M \xrightarrow{\varphi^*} N \otimes S(sW) \xrightarrow{\;/N^+ \otimes S(sW)\;} S(sW)$ maps $M^+$ to 0, which is equivalent to the condition $\varphi^*(M^+) \subset N^+ \otimes S(sW)$, defining $\Hom^0$. 
Thus, the adjunction \eqref{adj0} for $G = T^k$ implies 
that the left adjoint $\Trd_k(M)$ 
is the Sullivan minimal model of the right adjoint $\mc{T}^k X = \Map^0(T^k, X) \dslash T^k$.
\end{proof}

\section{Maximal parabolic subalgebras of type \texorpdfstring{$E_k$}{} Lie algebras}
\label{sec-para} 

This section provides a description of the maximal (proper) parabolic subalgebras of simple (semisimple for $k = 3$ and Kac-Moody for $k \ge 9$) Lie algebras $\g$ of type $E_{k(k)}$
and serves as 
preparation for the next section where we will have 
an action of those algebras on the toroidification model. 
The maximally noncompact, or ~{\it split}, real forms
$\mathfrak{g}\equiv \mathfrak{g}^\CC(\RR)$ behave closely to 
the underlying complex forms $\mathfrak{g}^\CC$ \cite{Warner} \cite{OV}. 
First, one can use the same basis as for the
corresponding complex (semi)simple (or Kac-Moody) Lie algebra $\mathfrak{g}^\CC$, but over $\RR$.
Second, the Satake diagram, which one needs for real forms, coincides with the Dynkin
diagram in the case of split forms. 
Parabolic subalgebras for split real forms of semisimple Lie algebras are described 
generally in \cite{OV}\cite{FH04}\cite{TY05} and explicitly for 
the exceptional Lie algebras
in \cite{Dobrev08}\cite{DM20}. 

\smallskip 
We start with generalities, applicable in the finite-dimensional case, then move on to the particulars of $E_{k(k)}$, $3\le k \le 8$, and later in this section, discuss a generalization to the Kac-Moody setting $k \ge 9$. In the case $k\le 2$, 
the setup is somewhat different; see the end of this section.

\paragraph{Root decomposition of the split real form of a semisimple Lie algebra.} Choosing a \emph{Cartan subalgebra} $\mathfrak{h}\subset\mathfrak{g}$, that is, a maximal abelian subalgebra of 
semisimple elements, gives the root space decomposition of $\mathfrak{g}$ into weight spaces of $\mathfrak{h}$:
\vspace{-4mm}
\begin{equation}
\label{root-decomp}
\mathfrak{g}=\mathfrak{h}\oplus \bigoplus_{\alpha\in\Delta} \mathfrak{g}_\alpha\,,
\end{equation}

\vspace{-2mm} 
\noindent where the root space $\mathfrak{g}_\alpha$ for an 
eigenvalue $\alpha \co \mathfrak{h}\to\RR$ is
the one-dimensional subspace
\vspace{-2mm} 
\(\
\label{g-alpha}
\mathfrak{g}_\alpha = 
\big\{ x\in\mathfrak{g}\; |\; [\,h\,,\,x\,] = \alpha(h) x\,
\;\;\textrm{for all} \;\; h\in\mathfrak{h}\big\}.
\)

\vspace{-2mm}
\noindent Here $\Delta$ is the set of \emph{roots}, i.e.,  the set of $\alpha\neq 0$ for which $\mathfrak{g}_\alpha\neq\{0\}$.  
In the set of roots $\Delta\subset\mathfrak{h}^*$, we choose a system of \emph{simple roots}
\vspace{-2mm} 
\(
\label{Pi}
\Pi=\{\alpha_1,\ldots,\alpha_k\},
\)
where $k=\dim \mathfrak{h}=\rank \mathfrak{g}$. Then any root $\alpha\in\Delta$ can be written as a nontrivial integral linear combination of the simple roots
$
\alpha =\sum_{i=1}^k m_i\alpha_i
$.
Here either all $m_i\geq0$, in which case $\alpha$ is called a \emph{positive root} ($\alpha>0$), or all $m_i\leq0$, in which case $\alpha$ is called
a \emph{negative root} ($\alpha<0$). The set of positive/negative roots is denoted by $\Delta^{\pm}$, and these satisfy $\Delta^-=-\Delta^+$. 
The  (upper/lower) \emph{Borel subalgebra} is
\vspace{-2mm} 
\begin{align}
\label{Borel}
\mathfrak{b}^\pm = \mathfrak{h}\oplus\mathfrak{n}^\pm
= \mathfrak{h} \oplus \bigoplus_{\alpha\in\Delta^\pm} \mathfrak{g}_\alpha\;.
\end{align}

\vspace{-2mm} 
\noindent  Generally, the nilpotent subalgebras 
 $\mathfrak{n}^\pm$ of $\mathfrak{g}$ can be thought of as strictly upper/lower triangular matrices, while $\mathfrak{h}$ would constitute diagonal matrices.

\medskip 
\paragraph{Maximal parabolics of the split real form of a semisimple Lie algebra.} A \emph{parabolic subalgebra} is a Lie subalgebra of $\g$ containing a Borel subalgebra, let us say, the upper one. A parabolic subalgebra can be described by choosing a subset $\Sigma$ of the set of simple roots \eqref{Pi}. The subset $\Sigma$ generates the set $\langle \Sigma \rangle$ of all roots which can be written as integral linear combinations of the roots in $\Sigma$.
The \emph{maximal} (proper) parabolic subalgebras are characterized by choosing  $\Sigma$ to be the set of all simple roots of $\mathfrak{g}$ but one: 
  $\Sigma=\Pi\setminus\{\alpha_{i}\}$. Thus, a maximal parabolic subalgebra corresponds to the choice of a node in the Dynkin diagram.
We have the decomposition into root subspaces:
\vspace{-2mm} 
\begin{equation}
    \mathfrak{ p }= \mathfrak{ h} \oplus \;\; \bigoplus_{\mathclap{\alpha \in \Delta(\mathfrak{ p)}}} \; \mathfrak{g}_\alpha\, , \qquad \text{ where }  
    \qquad \Delta(\mathfrak{ p}) = \Delta^+ \cup \langle \Sigma \rangle \, .
\end{equation}

\vspace{-2mm}
\noindent The parabolic subalgebra can also be decomposed into the direct
%
sum of a \emph{reductive
subalgebra} $\mathfrak{ l}$, called a \emph{Levi factor}, and 
the \emph{nilradical} $\mathfrak{ n}$ of $\p$:
\vspace{-2mm}
\begin{equation}
    \mathfrak{ p }= \mathfrak{ l} \oplus \mathfrak{ n} \coeq
    \bigg( \mathfrak{ h} \oplus \; \bigoplus_{\mathclap{\alpha \in \langle \Sigma \rangle}} \; \mathfrak{ g}_\alpha\bigg) \;
    \oplus \;
  \bigg( \quad \;  \bigoplus_{\mathclap{\alpha \in \Delta^+ \setminus \langle \Sigma \rangle}} \;\; \mathfrak{ g}_\alpha\bigg)
    \;.
\end{equation}

\vspace{-2mm} 
\noindent The Levi factor, which has the same rank as $\mathfrak{g}$, can further be  decomposed into
\vspace{-2mm}
\begin{equation}
\label{LanglandsLA}
\mathfrak{ l} = \mathfrak{ m} \oplus \mathfrak{ a}
\end{equation}

\vspace{-2mm} 
\noindent with 
\vspace{-2mm} 
\begin{itemize} 
\item $\mathfrak{a} \subset \mathfrak{ h}$ abelian, given by 
    $\mathfrak{ a} = \big\{h \in \mathfrak{h} \mid \alpha(h) = 0 \; \textrm{ for all } \alpha \in \Sigma \big\} 
        $;

\item  $\mathfrak{m} = [\, \mathfrak{l} \,, \, \mathfrak{l} \,]$ semisimple,
given by 
$  
\mathfrak{ m }= \mathfrak{ a}^\perp \oplus \bigoplus_{{\alpha \in \langle \Sigma \rangle}} \mathfrak{ g}_\alpha 
$,
where the orthogonal complement $\mathfrak{ a}^\perp$ is taken within $\mathfrak{h}$ with respect to the 
Killing
form $(- , -)$.

\end{itemize} 
 This leads to the \emph{Langlands decomposition} \cite{Knapp.1997}\cite{Knapp.2023} of the parabolic subalgebra
 \vspace{-2mm} 
\begin{equation}
\label{Langlands}
    \mathfrak{ p }= \mathfrak{ m} \oplus \mathfrak{ a} \oplus \mathfrak{n }\,,
\end{equation}

\vspace{-2mm} 
\noindent refining the \emph{Levi decomposition}, in which $\mathfr{m}$ is a \emph{Levi subalgebra} and $\mathfrak{ a} \oplus \mathfrak{n }$ is the \emph{radical} of $\p$.

\paragraph{Maximal parabolics of split real forms $\mathfr{e}_{k(k)}$ for $3 \le k \le 8$.}
The maximal parabolic subalgebras $\p_k^{k(k)}$ of $\mathfr{e}_{k(k)}$ we will use correspond  
  to the choice of the exceptional node $i=k$ in the Dynkin diagram 
  \eqref{Dynkin-schem} or \eqref{Dynkin-schem-3}, so that $\Sigma=\{\alpha_1, \cdots, \alpha_{k-1}\}$. In this case, 
 \vspace{-2mm} 
\begin{align*}
\mathfrak{a} = \mathfrak{so}(1,1)
 = \mathfrak{gl}(1,\R) \, ,
\end{align*}

\vspace{-2mm} 
\noindent 
and
\vspace{-1mm} 
\[
\mathfr{m} = \mathfr{sl}(k,\R),
\]
being the semisimple Lie algebra whose Dynkin diagram is obtained by removing the $k$th node from the Dynkin diagram of $\mathfrak{g}$ and leaving what is known in the physical context as
the ``gravity line'':
\vspace{-1mm} 
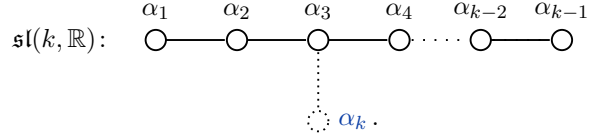
\[
\hspace{-1cm} 
\qquad
  \scalebox{.9}{$
  \raisebox{-30pt}{\begin{tikzpicture}[scale=.6]
     \foreach \x in {0,...,5}
    \draw[thick,xshift=\x cm] (\x cm,0) circle (2.5 mm);
    \foreach \y in {0, 1,2, 4}
    \draw[thick,xshift=\y cm] (\y cm,0) ++(.25 cm, 0) -- +(14.5 mm,0);
    \foreach \y in {3,4}
    \draw[loosely dotted, thick,xshift=\y cm] (\y cm,0) ++(.3 cm, 0) -- +(14 mm,0);
    \draw[dotted, thick] (4 cm, -2 cm) circle (2.5 mm);
    \draw[dotted, thick] (4 cm, -3mm) -- +(0, -1.45 cm);
    \node at (0,.7) {$\alpha_1$};
    \node at (2,.7) {$\alpha_2$};
    \node at (4,.7) {$\alpha_3$};
    \node at (6,.7) {$\alpha_4$};
    \node at (8,.7) {$\alpha_{k-2}$};
    \node at (10,.7) {$\alpha_{k-1}$};
               \node at (5,-2) {{\color{darkblue} $\alpha_k\,$}.};
     \node at (-2.2,0) {
     \bf $\mathfrak{sl}(k, \RR) \co$};
  \end{tikzpicture}
  }
  $}
\]

\vspace{-4mm}
\paragraph{Split real forms $\mathfr{e}_{k(k)}$ of Kac-Moody algebras for $k \ge 3$.}
The structure theory for parabolic subalgebras of Kac-Moody algebras $\mathfr{e}_{k(k)}$ for $k \ge 9$ is not well explored, but may be developed very much like in the finite-dimensional case, $3 \le k \le 8$, though the setup is somewhat different. The choice of a Cartan subalgebra $\h$, simple roots, and a (generalized) Cartan matrix is part of the Kac-Moody algebra structure, see \cite[Chapter 1]{Kac}. 
For $k \ge 3$, the \emph{generalized Cartan matrix of type $E_k$}
\begin{equation}
\label{Cartan-matrix}
C =
{\footnotesize
\begin{bmatrix}
  2 & -1 & 0 & 0 & 0 & \dots & 0 & 0 & 0\\
  -1 & 2 & -1 & 0 & 0 & \dots & 0 & 0 & 0\\
0 & -1 & 2 & -1 & 0 & \dots & 0 & 0 & -1\\
0 & 0 & -1 & 2 & -1 & \dots & 0& 0 & 0\\
0 & 0 & 0 & -1 & 2 & \dots & 0 & 0 & 0\\
& & & & & \ddots\\
0 & 0 & 0 & 0 & 0 & \dots & 2 & -1 & 0\\
0 & 0 & 0 & 0 & 0 & \dots & -1 & 2 & 0\\
0 & 0 & -1 & 0 & 0 & \dots & 0 & 0 & 2
\end{bmatrix}
}
\end{equation}
is effectively the negative incidence matrix of the Dynkin diagram: apart from the
diagonal entries $c_{ii} = 2$, one has
\[
c_{ij} =
\begin{cases}
  -1 & \text{if $\alpha_i$ and $\alpha_j$ are joined by an edge},
  \\[-2pt]
  \phantom{-}0 & \text{otherwise}.
\end{cases}
\]
For $k \ne 9$, the real Kac-Moody algebra $\mathfr{e}_{k(k)}$ is generated by (\emph{Chevalley}) generators
$e_1,e_2, \dots, e_{k}, \linebreak[0] f_1, f_2, \dots, f_{k}$, and $h_1, h_2, \dots, h_k$
and relations
\begin{itemize}
\setlength\itemsep{-1pt}
\item $[h_i,h_j] = 0$ \qquad \quad  \;\; for $1 \le 
i,j \le k$;
\item $[h_i,e_j] = c_{ij} e_j$;
\item $[h_i,f_j] = - c_{ij} f_j$;
\item $[e_i, f_j] = \delta_{ij} h_i$;
\item If $i \ne j$ (so $c_{{ij}}\leq 0$) then $\operatorname{ad}
  (e_i)^{1-c_{ij}}(e_j) = 0$ and $\operatorname{ad}
  (f_{i})^{1-c_{ij}}(f_{j})=0$ (\emph{Serre relations}).
  \end{itemize}
The \emph{Cartan subalgebra} is
\vspace{-2mm}
\[
\h = \bigoplus_{i=1}^k \R h_i,
\]
while the \emph{Borel subalgebra} $\mathfr{b}^+$ is the subalgebra generated by $h_i$'s and $e_i$'s, $i = 1, \dots, k$; similarly for $\mathfr{b}^-$.

\smallskip 
For $3 \le k \le 8$,  when $\det C = 9 - k > 0$, this gives the split real forms $\mathfr{e}_{3(3)} = \mathfr{sl}(3, \RR) \, \times \,\mathfr{sl}(2, \RR)$,  $\mathfr{e}_{4(4)} = \mathfr{sl}(5, \RR)$,
$\mathfrak{e}_{5(5)} =\mathfr{so}(5,5)$,  
$\mathfrak{e}_{6(6)}$, $\mathfrak{e}_{7(7)}$, $\mathfrak{e}_{8(8)}$ of semisimple Lie algebras. For $k \ge 10$, when $\det C < 0$, we get the split real forms $\mathfrak{e}_{k(k)}$ of Kac-Moody algebras of \emph{indefinite or Lorentzian type}, among which $\mathfrak{e}_{10(10)}$ is classified as \emph{hyperbolic}. See \Cref{table1}.

\newpage 

\smallskip 
For $k=9$, when $\det C = 0$, the above construction gives the derived subalgebra $[\mathfrak{e}_{9(9)}, \mathfrak{e}_{9(9)}]$, also known as the \emph{affine Lie algebra} corresponding to the finite-dimensional simple Lie algebra $\mathfrak{e}_{8(8)}$, of the \emph{affine Kac-Moody algebra} $\tilde{\mathfrak{e}}_{8(8)} = \mathfrak{e}_{9(9)}$. The full affine Kac-Moody algebra $\mathfrak{e}_{9(9)}$ may be constructed using the notion of a realization of the generalized Cartan matrix as $\g_9 = \mathfrak{e}_{9(9)}$ in \Cref{parabolic_action}, see particularly \Cref{identification}(ii). The Cartan subalgebra $\h$ of $\mathfrak{e}_{9(9)}$ adds an extra generator to $\bigoplus_{i=1}^k \R h_i$, the Lie algebra $\mathfrak{e}_{9(9)}$ is generated by $e_i$'s, $f_i$'s, and the elements of $\h$, and the Borel subalgebra $\mathfr{b}^+$ is generated by $e_i$'s, $i = 1, \dots, k$, and the elements of $\h$.

\smallskip
The \emph{simple roots} $\Pi = \{\alpha_1, \dots, \alpha_k\} \subset \h^*$ are defined by
\vspace{-1mm}
\[
\alpha_{i}(h_{j}) = c_{ji},
\]

\vspace{-1mm} 
\noindent with the condition of vanishing on the extra generator of $\h$ in the case $k=9$.

\paragraph{Maximal parabolics of Kac-Moody algebras $\mathfr{e}_{k(k)}$ for $k \ge 3$.} The Dynkin diagram has the shape \eqref{Dynkin-schem} for all $k \ge 4$ and \eqref{Dynkin-schem-3} for $k = 3$. We still have a root decomposition \eqref{root-decomp}, with the positive roots $\Delta^+$ being those roots which are integral linear combinations of the simple roots $\alpha_i$, $i = 1, \dots, k$, but the root spaces may no longer be one-dimensional. The Borel subalgebras also have a root-space decomposition as in \eqref{Borel}.

We consider the maximal parabolic subalgebra $\p_k^{k(k)}$ corresponding to the subset $\Sigma = \Pi \setminus \{\alpha_k\}$ for each $k$, as above. The parabolic subalgebra $\p_k^{k(k)}$ is generated by $e_1, \dots, e_{k}$, $f_1, \dots, f_{k-1}$, and the elements of the Cartan subalgebra $\h$. The Langlands decomposition \eqref{Langlands} still takes place, with $\mathfr{a} = \mathfr{so}(1,1)$ being one-dimensional, except for $k=9$, when $\dim \mathfr{a} = 2$. The Levi subalgebra $\mathfr{m}$, being the subalgebra generated by $e_1, \dots, e_{k-1}$, $f_1, \dots, f_{k-1}$ and $h_1, \dots, h_{k-1}$, is the Kac-Moody algebra corresponding to the Cartan matrix equal to the leading principal $(k-1) \times (k-1)$ submatrix of $C$. This submatrix is the Cartan matrix of type $A_{k-1}$, and therefore, the ``gravity line'' Lie algebra $\mathfr{m}$ is $\mathfr{sl}(k,\R)$ and thereby finite-dimensional for all $k \ge 3$. Given that $\dim \mathfr{e}_{k(k)} = \infty$ for $k \ge 9$, we see that the nilradical $\mathfr{n}$ also has to be infinite-dimensional.

\paragraph{Parabolic subalgebras in string theory and M-theory.} 
The maximal parabolic subalgebra $\p_k^{k(k)}$ appears in string theory and M-theory calculations \cite{GRV10a}\cite{GRV10b}.
Such parabolics persist in the low-energy limit,  in which the normalized volume\footnote{Here $\ell_{11}$ is the Planck length or scale.} $V_{k}/\ell_{11}^{k}$ of the M-theory torus
 is large,
            the semiclassical limit, which is eleven-dimensional supergravity, becomes a good approximation.

\smallskip 
Chern-Simons couplings play a role in the appearance of enhanced symmetries of Cremmer--Julia type in various theories. For generic values of the Chern--Simons coupling, there is only a parabolic Lie  subalgebra of symmetries after dimensional reduction, but this gets enhanced to the full and larger Cremmer--Julia Lie algebra of hidden symmetries if the coupling takes a specific value $\pm 1$ \cite{HKL16}. 

\smallskip 
By establishing an action of the maximal parabolic, we 
establish
a symmetry-breaking pattern.

\smallskip 
The following statements on the structure of maximal parabolics come from \cite{DM20} for $4 \le k \le 8$. For $k \le 3$, things work differently but are simpler, and we present these cases below for completeness. These computations make up the bulk of \Cref{table1}.

\medskip 
\noindent \fbox{$k=8$}  The split real form of $\mathfr{e}_8$ is $\mathfrak{e}_{8(8)}\,$.
The maximal  parabolic subalgebras   are given by their Langlands decomposition:
$$
\mathfrak{p}^{8(8)}_i = \mathfrak{m}^{8(8)}_i\ \oplus\ 
\mathfrak{so}(1,1)\ \oplus\ \mathfrak{n}^{\,8(8)}_i\ , \quad i=1,\ldots, 8 \;,
$$
where for $i = 8$
$$
 \begin{array}{ll}
\mathfrak{m}^{8(8)}_8  = \mathfrak{sl}(8,\RR)\;, \qquad 
& \dim\,\mathfrak{n}_8^{\,8(8)} ~=~ 92\,.
\end{array}
$$
Note   the dimensions in relation to $\mathfrak{e}_{8(8)}$ check the equality  $248=92+63+1+92$:
 $$
 \mathfrak{e}_{8(8)}=  \mathfrak{n}^{\,8(8) -}_8 \oplus \mathfrak{p}^{8(8)}_8
  = \mathfrak{n}^{\,8(8) -}_8 \oplus \mathfrak{m}^{8(8)}_8\ \oplus\ 
\mathfrak{so}(1,1)\ \oplus\ \mathfrak{n}^{\,8(8)}_8 .
 $$
The case $k=8$ indeed gives a global symmetry
$
\mathfrak{gl}(8, \RR) \ltimes (\RR^{56} \oplus \RR^{28} \oplus \RR^8)
$,
where the 56 shifts come from the direct scalars descending from the 3-form,
the 28 shifts come from scalars arising from the dual 6-form, and the 8 shifts come from the dual graviton. 
We observe that this gives the nilradical $\mathfrak{n}_8^{8(8)}$.


\medskip 
\noindent \fbox{$k=7$} For the split real form $\mathfrak{e}_{7(7)}$ of $\mathfr{e}_8$,
the maximal parabolic subalgebras are determined by:
$$
  \mathfrak{p}^{7(7)}_i = \mathfrak{m}^{7(7)}_i\ \oplus\ \mathfrak{so}(1,1)\ 
 \oplus\ \mathfrak{n}^{\,7(7)}_i\ , \quad i=1,\ldots,7\, ,
 $$
where for $i = 7$
 $$
 \begin{array}{ll}
  \mathfrak{m}^{7(7)}_7 = \mathfrak{sl}(7,\RR)\;,
& \dim\,\mathfrak{n}_7^{\,7(7)} ~=~ 42\;.
\end{array} 
$$

\medskip 
\noindent \fbox{$k=6$} For the split real form $\mathfrak{e}_{6(6)}$ of $\mathfr{e}_6$, 
the maximal parabolic subalgebras are determined by the decomposition:
$$
 \mathfrak{p}^{ 6(6)}_i ~=~ \mathfrak{m}^{ 6(6)}_i\ \oplus\ \mathfrak{so}(1,1)\ \oplus\
 \mathfrak{n}^{\, 6(6)}_i\ , \quad i=1, \dots, 6\, ,
$$
where  for $i = 6$
$$
 \begin{array}{ll}
\mathfrak{m}^{ 6(6)}_6 = \ \mathfrak{sl}(6,\RR)\;,
&  \dim\,\mathfrak{n}_6^{\, 6(6)} ~=~ 21\;.
\end{array} 
$$

\medskip 
\noindent \fbox{$k=5$} For $\mathfrak{e}_{5(5)}=\mathfrak{so}(5,5)$, the maximal parabolic subalgebras are determined by:
$$
 \mathfrak{p}^{ 5(5)}_i ~=~ \mathfrak{m}^{ 5(5)}_i\ \oplus\ \mathfrak{so}(1,1)\ \oplus\
 \mathfrak{n}^{\, 5(5)}_i\ , \quad i=1, \dots, 5\, ,
$$
where  for $i = 5$
$$
 \begin{array}{ll}
 \mathfrak{m}^{ 5(5)}_5  =  
\mathfrak{sl}(5,\RR) \;,
& \dim\,\mathfrak{n}_5^{\, 5(5)} ~=~10 \;.
\end{array} 
$$


\medskip 
\noindent \fbox{$k=4$}  For $\mathfrak{e}_{4(4)}=\mathfrak{sl}(5, \R)$, the maximal parabolic subalgebras are  
$$
 \mathfrak{p}^{ 4(4)}_i ~=~ \mathfrak{m}^{ 4(4)}_i\ \oplus\ \mathfrak{so}(1,1)\ \oplus\
 \mathfrak{n}^{\, 4(4)}_i\ , \quad i=1, \dots, 4\, ,
$$
where  for $i = 4$
$$
 \begin{array}{ll}
 \mathfrak{m}^{ 4(4)}_4   = \mathfrak{sl}(4, \R)\,, 
& \dim\,\mathfrak{n}_4^{\, 4(4)} ~=~4 \;.
\end{array} 
$$


\medskip 
\noindent \fbox{$k=3$}  For $\mathfrak{e}_{3(3)}= \mathfrak{sl}(3, \R) \times \mathfrak{sl}(2, \R)$,
we have two factors.
The algebra $\mathfrak{sl}(3, \RR)$ has two simple roots $\alpha_1$, $\alpha_2$ and positive roots
given by $\Delta^+(\mathfrak{sl}(3, \RR))=\{\alpha_1, \alpha_2, \alpha_1 + \alpha_2\}$. The algebra $\mathfrak{sl}(2, \RR)$ has one simple root $\alpha_3$, the only positive root. The maximal parabolic 
corresponding to $\alpha_3$ is
\[
\mathfrak{p}_3^{3(3)} = \mathfrak{m}^{ 3(3)}_3\ \oplus\ \mathfrak{so}(1,1)\ \oplus\
 \mathfrak{n}^{\, 3(3)}_3\, ,
\]
where
$$
 \begin{array}{ll}
 \mathfrak{m}^{ 3(3)}_3   = \mathfrak{sl}(3, \R)\,, 
& \dim\,\mathfrak{n}_3^{\, 3(3)} ~=~1 \;.
\end{array} 
$$



\medskip 
\noindent \fbox{$k=2$} For $\mathfrak{e}_{2(2)}=\mathfrak{sl}(2, \R)$, the root $\alpha_2$ is missing, i.e., $\Sigma = \Pi = \{\alpha_1\}$, and the maximal parabolic subalgebra corresponding to $\alpha_2$ is nonproper: $\p_2^{2(2)} = \mathfrak{sl}(2, \R) = \mathfrak{e}_{2(2)}$.

\medskip 
\noindent \fbox{$k=1$} For $\mathfrak{e}_{1(1)}=\mathfrak{sl}(1, \R) = 0$, there are no roots whatsoever, i.e., $\Sigma = \Pi = \varnothing$, and the maximal parabolic subalgebra is $\p_1^{1(1)} = 0 = \mathfrak{e}_{1(1)}$.

\medskip 
\noindent \fbox{$k=0$} For $\mathfrak{e}_{0(0)}=\mathfrak{sl}(0, \R) = \varnothing$, the maximal parabolic subalgebra is $\p_0^{0(0)} = \varnothing = \mathfrak{e}_{0(0)}$.

\section{The action of a parabolic subalgebra \texorpdfstring{$\p_k$}{}}
\label{parabolic_action}

In this section, we establish the action of (trivial central extensions of) the distinguished maximal parabolic subalgebras $\p^{k(k)}_k$
described in \cref{sec-para} on the homotopy type of the toroidification models  $\mc{T}^k S^4 := \mc{L}^k S^4\dslash T^k$
described in \cref{Sec-toroidMod}. One physical significance of this action is that it corresponds to being
on-shell, i.e., capturing physical states.  We will assume $k \ge 3$ here and deal with the low-$k$ cases, $0 \le k \le 2$, separately afterward.

\medskip 
The plan of the construction follows the work \cite{SV1}, which creates a large enough, in fact maximal, split real torus acting on the Sullivan minimal model $M(\mc{L}_c^k S^4)$ of the iterated cyclification of $S^4$. The same torus turns out to be the maximal split torus acting on the Sullivan minimal model $M(\mc{T}^k S^4)$ of the $k$-toroidification of $S^4$, as per \Cref{maxtorus} below. The action of the split torus decomposes the vector space $M(\mc{L}_c^k S^4)$ into weight subspaces, see the proof of the theorem. The parabolic subalgebra $\p_k$ we will describe will contain the Lie algebra $\h_k$  of the torus as the Cartan subalgebra. In a Lie algebra action, weights add. This dramatically restricts the possibilities for the action of $\p_k$ on $M(\mc{T}^k S^4)$. Combined with our focus on linear actions, the action formulas which we present later in this section are predetermined, up to constant factors.

\begin{theorem}[Maximal torus of toroidification]
\label{maxtorus}
The real algebraic group $\Aut M (\mc{T}^k S^4)$ for $k \ge 0$ contains a maximal $\RR$-split torus canonically isomorphic to $\GG_m (\R)\times T^{k}_\CC (\R)$, where $\GG_m(\R)  = \R^\times$ is the multiplicative group of nonzero real numbers and $T^{k}_\CC (\R) \cong (\R^\times)^k$ is the split form of the complexification $T^{k}_\CC$ of the compact torus $T^k$ used in the definition of $\mc{T}^k S^4$.
\end{theorem}

\begin{proof}
This theorem appears to be similar to that of \cite[Corollary 3.6]{SV1}, but the proof we presented there was based on the computation $\ker d \cap V = k+1$, which does not hold in our case. We will give a less conceptual, though simpler proof, which actually works in both cases.

The fact that $\Aut M (\mc{T}^k S^4)$ is a real algebraic group follows from the following argument. An endomorphism of $M (\mc{T}^k S^4)$ as a graded algebra is determined by the images of the finitely many free generators of $M (\mc{T}^k S^4)$. Each generator may map anywhere to the component of the corresponding degree, and all such components are finite-dimensional. This defines a real affine space. The compatibility conditions with the differential define a subvariety in this affine space, while the invertibility condition for the endomorphism defines a Zariski open set in the subvariety.

The $\Q$-split algebraic torus $\GG_m(\Q) \times T^{k}_\CC (\Q)$ acts on $M (\mc{T}^k S^4)$: an element $(p_0/q_0, g)$ acts by a map $\mc{T}^k S^4 \to \mc{T}^k S^4$ induced on $\mc{T}^k S^4 = \Map (T^k, S^4)$ by postcomposion with a map $S^4 \to S^4$ of degree $p_0/q_0$ and precomposition with the map $g: T^k \to T^k$ in the rational category. Let us decsribe how this works. Since a map $S^4 \to S^4$ of degree $q_0 \ne 0$ is invertible in the rational category of spaces, morphisms of fractional degrees make sense. The $\Q$-split torus $T^{k}_\CC (\Q)$ acts on the compact torus $T^k$ rationally in a similar manner, as follows. First, choose an isomorphism $T^k \cong (S^1)^k$. This produces a splitting $T^k_\CC (\CC) \cong (\CC^\times)^k$, which yields a splitting $T^k_\CC (\Q) \cong (\Q^\times)^k$. Then we make an element $(p_1/q_1, \dots, p_k/q_k) \in (\Q^\times)^k$ act on $\Map (T^k, S^4)$ by precomposition with a map $T^k \to T^k$ of multidegree $(p_1/q_1, \dots, p_k/q_k)$ in rational category. The induced action on the Sullivan minimal model $M (\mc{T}^k S^4)$ is algebraic, because the construction comes from an action at the level of spaces in the rational category. Thus, we get a morphism of real algebraic groups $\GG_m(\R) \times T^{k}_\CC (\R) \to \Aut M (\mc{T}^k S^4)$ defining an action of the $(k+1)$-dimensional split real torus.

We claim that via this action, the torus $\GG_m(\R) \times T^{k}_\CC (\R)$ embeds in $\Aut M (\mc{T}^k S^4)$ as a maximal split torus. Let us recall general facts about split torus actions on DGCAs.

If a split real torus $T$ acts on a graded real vector space with finite-dimensional components, such as $M = S(V)$, that vector space splits into a \emph{weight decomposition}
\vspace{-2mm} 
\begin{equation*}
M = \bigoplus_{\alpha \in \mathfrak{X}(T)} M_\alpha
\end{equation*}
indexed by the \emph{character group} $\mathfrak{X}(T) = \Hom_\RR (T, \GG_m)$ of
real algebraic group homomorphisms from $T$ to the multiplicative group
$\GG_m$, so that $T$ acts on each \emph{weight space} $M_\alpha$ by
the character $\alpha$:
\[
M_\alpha = \big\{m \in M \; | \; t \cdot m = \alpha(t) m \quad \text{for
  all $t \in T$}\big\}.
\]
If the rank (dimension) of $T$ is  $r \ge 0$, then $\mathfrak{X}(T) \cong
\ZZ^r$; see \cite[Corollary 8.2 and Proposition 8.5]{borel}. Since
automorphisms of $M$ have to respect the DGCA structure, i.e.,
the $\ZZ$-grading, differential, and multiplication on $M$, the weight
decomposition is automatically compatible with it. That is, we have
\vspace{-1mm} 
  \begin{enumerate}[{\bf (i)}]
  \setlength\itemsep{-1pt}
  \item $M^n = \bigoplus_{\alpha \in \mathfrak{X}(T)} (M_\alpha \cap M^n)$ for
    each $n \ge 0$;
  \item $d(M_\alpha ) \subset M_\alpha$ for each $\alpha \in \mathfrak{X}(T)$;
    \item $M_\alpha \cdot M_\beta \subset M_{\alpha+\beta}$ for all
      $\alpha, \beta \in \mathfrak{X}(T)$.
    \end{enumerate}
Conversely, a weight decomposition that satisfies these properties defines an action of the torus on the DGCA $M$.

We argued in \cite{SV1}[Section 3.4] that a very similar action of $\GG_m(\R)^{k+1}$ on the cyclification $\mc{L}_c^k S^4$ corresponds to the action on its Sullivan minimal model $M(\mc{L}_c^k S^4)$ that satisfies
\vspace{-2mm}
\begin{align}
\label{toract1}
t g_4 & = \epsilon^{-1}_0(t) g_4, & t g_7 &= \epsilon^{-2}_0(t) g_7,\\
\label{toract2}
t w_i & =  \epsilon_i^{-1}(t) w_i, & [t,s_i] & \coeq t s_i - s_i t =  \epsilon_i(t) s_i,
\end{align}
where
\vspace{-2mm}
\begin{equation}
\label{epsilon}
\begin{split}
    \epsilon_0(t_0, t_1, \dots, t_k) & = t_0^{-1},\\
    \epsilon_i(t_0, t_1, \dots, t_k) & = t_i^{-1}
\end{split}
\end{equation}
for $i = 1, \dots, k$. The generators of $M(\mc{L}_c^k S^4)$ are exactly the same as those of $M(\mc{T}^k S^4)$, i.e., the underlying graded algebras are isomorphic, but the differentials are slightly different, see \cite[Section 2.5]{SV1} ($M(\mc{L}_c^k S^4)$ is just an iteration of $M(\mc{L}_c S^4)$, see \Cref{CyclicModel}) and \Cref{tor}, respectively. Here we use a slightly different normalization as compared to \cite{SV1}, whence the excessive inverses, but this is more convenient for the purposes of this paper.

The same formulas correspond to the action, defined in the rational homotopy category above, of the split torus $\GG_m(\R) \times T^{k}_\CC (\R)$, under a choice of an isomorphism $T^k \cong (S^1)^k$ and thereby splitting $ T^{k}_\CC (\R) \cong \GG_m(\R)^k$. This action produces a weight decomposition of $M = M(\mc{T}^k S^4)$. For example, $s_1 s_2 g_7 \cdot w_2 \in M_{\epsilon_0^{-2} \epsilon_1}$.

From \eqref{toract1}--\eqref{epsilon}, it is clear that if an element $t$ of the torus $\GG_m(\R) \times T^{k}_\CC (\R)$ acts trivially on $w_1, \dots, w_k$, and $g_4$, then $t=1$. Therefore, the action homomorphism $\GG_m(\R) \times T^{k}_\CC (\R) \to \Aut M(\mc{T}^k S^4)$ is injective, and we have got a split real torus of rank $k+1$ in $\Aut M(\mc{T}^k S^4)$.

Now let us prove that this split torus is maximal. If we have an element $t' \in \Aut M(\mc{T}^k S^4)$ which commutes with every element of the split torus $\GG_m(\R) \times T^{k}_\CC (\R)$, the automorphism $t'$ will automatically preserve the weight spaces of the split torus action on $M = M(\mc{T}^k S^4)$. Observe that the following weight subspaces are one-dimensional: $M_{\epsilon_0^{-1}} \cap M^4 = \R g_4$, $M_{\epsilon_i^{-1}} = \R w_i$, $i = 1, \dots, k$. Hence, we have $t' g_4 = \beta_0 g_4$ and $t' w_i = \beta_i w_i$ for all $i = 1, \dots, k$ and some $\beta_0, \beta_1, \dots, \beta_k \in \R \setminus \{0\}$. We claim that $t'$ has to act on the generators of $M$ according to Equations \eqref{toract1}--\eqref{toract2} for $t = (\beta_0, \beta_1, \dots, \beta_k)$. This will identify this element of the split torus $\GG_m(\R) \times T^{k}_\CC (\R) \subset \Aut M$ with the automorphism $t'$, and imply that the split torus is maximal.

It remains to establish the action of $t'$ on the generators of $M$, as claimed. We have $M_{\epsilon_0^{-2}} \cap M^7 = \R g_7$ and $M_{\epsilon_0^{-2} \epsilon_i} \cap M^6 = \R s_i g_7$ for each $i$, whence $t' g_7 = \gamma g_7$ and $t' s_i g_7 = \gamma_i s_i g_7$ for some $\gamma, \gamma_i \in \R \setminus \{0\}$. We also have $dg_7 = - \tfrac{1}{2} g_4^2 + \sum_{i=1}^k w_i \cdot s_i g_7$, as per \Cref{tor}, and 
\vspace{-2mm} 
\begin{align*}
d t' g_7 & = d \gamma g_7 = - \tfrac{1}{2} \gamma g_4^2 + \gamma \sum_{i=1}^k w_i \cdot s_i g_7,\\[1pt]
t' dg_7 & = - \tfrac{1}{2} \beta_0^2 g_4^2 + \beta_i \gamma_i \sum_{i=1}^k w_i \cdot s_i g_7.
\end{align*}
Being an automorphism of the DGCA $M$, the element $t'$ has to commute with the differential, which implies that the right-hand sides of the above equations must be equal. Since $g_4$, $w_i$, and $s_i g_7$ are among the free generators of the DGCA $M$, this means $\gamma = \beta_0^2$ and $\gamma_i = \beta_0^2 \beta_i^{-1}$. Similarly, looking at the action of $t'$ on $d g_4$, $d s_i g_4$, etc., we show that $t' s_i g_4 = \beta_0 \beta_i^{-1} s_i g_4$, $t' s_i s_j g_4 = \beta_0 \beta_i^{-1} \beta_j^{-1} s_i s_j g_4$, and so on for all the generators of $M$. This confirms that $t'$ acts on $M$ as prescribed by Equations \eqref{toract1}--\eqref{toract2} for $t = (\beta_0, \beta_1, \dots, \beta_k)$ and concludes the proof.
\end{proof}

\medskip 
Below, we describe a (trivial) one-dimensional central
extension $\g_k$ of the split real form $\mathfrak{e}_{k(k)}$ of the complex Lie algebra $\mathfr{e}_k$ of type $E_k$
based on the
Chevalley generators of $\mathfr{e}_k$ as a \emph{Kac-Moody algebra} for $ k \ge 3$ except $k = 9$. For $k=9$, we will have $\g_9 = \mathfrak{e}_{9(9)}$, which is already a nontrivial central extension of the loop algebra $\mathfr{e}_{8(8)} \otimes \R[t, t^{-1}]$ augmented by a derivation \cite[Chapter 7]{Kac}. The theory of real forms of complex Kac-Moody algebras 
(see \cite{KW}\cite{BBMR}\cite{BenM})
generalizes the theory of
real semisimple Lie algebras
(see \cite{Knapp.1997}). The split real affine Kac-Moody algebras $\mathfr{e}_{9(9)}$, $\mathfr{e}_{10(10)}$, and $\mathfr{e}_{11(11)}$ were proposed as symmetries in the context of M-theory
(see \cite{HKN}\cite{DHN}\cite{West01}).

\medskip 
The maximal parabolic subalgebra $\p_k \subset \g_k$ we will define will also be a trivial central extension of the parabolic subalgebra $\p_k^{k(k)} \subset \mathfr{e}_{k(k)}$, except that for $k= 9$, we will just have $\p_9 = \p_9^{9(9)}$. In \Cref{sec-para}, we described the parabolic subalgebras $\p_k^{k(k)} \subset \mathfr{e}_{k(k)}$ explicitly in the classical, finite-dimensional case $k \le 8$. For $k \ge 9$, the parabolic subalgebra $\p_k^{k(k)}$ of the split real Kac-Moody algebra $\mathfr{e}_{k(k)}$ is defined by dropping the generator $f_k$ in its set of Chevalley generators, just as we do it in \Cref{pk-acts} below.

\medskip 
Let $\h_k$ be the $(k+1)$-dimensional real vector space with a basis
\(
\label{basis-h}
\{h_0, h_1, \dots, h_k\}
\)
and the Minkowski inner product\footnote{Here we use the negative of the inner product defined in \cite{SV1}\cite{SV2}, to match with the positive definiteness of the Killing form on the Cartan subalgebra of $\mathfr{e}_{k(k)}$ for $3 \le k \le 8$.}
\[
(h_0, h_0) = -1, \qquad (h_i, h_i) = 1 \quad \text{for $1 \le i \le k$,} \quad \text{and} \quad (h_i, h_j) = 0 \quad \text{for $0 \le i \ne j \le k$.}
\]
Let $\h_k^*$ be the linear dual space and $\{ \eps_0, \eps_1, \dots, \eps_k\}$ be the image of the basis \eqref{basis-h} under the isomorphism $\h_k \xrightarrow{\sim} \h_k^*$ induced by the inner product. Thus, we have
\[
\eps_0 (h_0) = -1, \qquad \eps_i (h_i) = 1 \quad \text{for $1 \le i \le k$,} \quad \text{and} \quad \eps_i (h_j) = 0 \quad \text{for $0 \le i \ne j \le k$.}
\]
The inner product on $\h_k$ induces one on the dual space $\h_k^*$. Explicitly,
\[
(\eps_0, \eps_0) = -1, \qquad (\eps_i, \eps_i) = 1 \quad \text{for $1 \le i \le k$,} \quad \text{and} \quad (\eps_i, \eps_j) = 0 \quad \text{for $0 \le i \ne j \le k$.}
\]

It is known (see \cite{Man}\cite{Dolgachev}) that the following set of vectors in the orthogonal complement in the dual space
\[
(K_k^\perp)^* \coeq \{ \alpha \in \h_k^* \mid \alpha(K_k) = 0\} \subset \h_k^*
\]
to the distinguished element
\[
K_k \coeq -3h_0 + \sum_{i=1}^k h_i \in \h_k
\]
forms a set of \emph{simple roots} of type $E_k$:
\begin{align*}
    \alpha_i & \coeq \eps_{i} - \eps_{i+1}, \qquad i = 1, \dots, k-1,\\
    \alpha_k & \coeq \eps_0 - \eps_1 - \eps_2 - \eps_3.
\end{align*}
The vectors
\vspace{-2mm} 
\begin{align*}
   \alpha_i^\vee & \coeq h_{i} - h_{i+1}, \qquad i = 1, \dots, k-1,\\
    \alpha_k^\vee & \coeq h_0 - h_1 - h_2 - h_3
\end{align*}
form the corresponding set of \emph{simple coroots} in the orthogonal complement $K_k^\perp =  \{ x \in \h_k \mid (x, K_k) = 0\} \subset \h_k$.
Then the matrix $C = (c_{ij})$, $1 \le i, j \le k$, with entries
\[
c_{ji} \coeq
\alpha_{i}(\alpha _{j}^{\vee})
\]
is exactly the generalized Cartan matrix \eqref{Cartan-matrix} of type $E_k$.

\medskip 
The root $\alpha_k$
refers to the distinguished node of the Dynkin
diagram \eqref{Dynkin-schem} of type $E_k$ (which is the same as the Satake diagram 
which describes the split real form $\mathfr{e}_{k(k)}$ of the complex simple Lie algebra $\mathfr{e}_k$ of type $E_k$).

Define the real Lie algebra $\g_k$ by (\emph{Chevalley}) generators
$e_1,e_2, \dots, e_{k}, \linebreak[0] f_1, f_2, \dots, f_{k}$ and the
elements of $\h_k$ and relations 
\cite{Ch55} (see \cite{Ca93} for an exposition):
\begin{itemize}
\setlength\itemsep{-1pt}
\item $[h,h'] = 0$ \qquad \quad  \;\; for $h, h' \in \h_k$;
\item $[h,e_i] = \alpha_i(h) e_i$  \quad \; for $h \in \h_k$;
\item $[h,f_i] = - \alpha_i(h) f_i$  \;\;  for $h \in \h_k$;
\item $[e_i, f_j] = \delta_{ij} \alpha_i^\vee$;
\item If $i \ne j$ (so $c_{{ij}}\leq 0$) then $\operatorname{ad}
  (e_i)^{1-c_{ij}}(e_j) = 0$ and $\operatorname{ad}
  (f_{i})^{1-c_{ij}}(f_{j})=0$ (\emph{Serre relations}).
  \end{itemize}

  \begin{prop}[Identification of the Lie algebra $\g_k$]
  $\,$
\label{identification}

  \vspace{-4mm} 
     \begin{enumerate}[\bf (i)]
     \setlength\itemsep{-2pt}
         \item For $3 \le k \le 8$, the Lie algebra $\g_k$ is the trivial one-dimensional central extension of the split real form $\mathfr{e}_{k(k)}$ of the finite-dimensional complex semisimple Lie algebra of type $E_k$.
         \item For $k = 9$, the Lie algebra $\g_k$ is the split real form $\mathfr{e}_{9(9)}$ of the Kac-Moody algebra of affine type $\tilde{E}_{8} = E_9$.
         \item For $k \ge 10$, the Lie algebra $\g_k$ is the trivial one-dimensional central extension of the split real form $\mathfr{e}_{k(k)}$ of the Kac-Moody algebra of indefinite type $E_k$.
     \end{enumerate}
  \end{prop}

\begin{proof}
(i) For $3 \le k \le 8$, note that 
$$
\big(K_k^\perp, \{\alpha_i\}_{i=1}^k, \{\alpha_i^\vee \}_{i=1}^k\big)
$$ 
is a \emph{realization} \cite[\S 1.1]{Kac} over $\RR$ of the generalized Cartan matrix $C$ of type $E_{k}$. Therefore, the real Lie algebra generated by $e_1, \dots, e_{k}, \linebreak[0] f_1, \dots, f_{k}$ and the elements of $K_k^\perp$ with the same defining relations as for $\g_k$ is the split real Kac-Moody algebra corresponding to the Cartan matrix. In this case, this Kac-Moody algebra is the finite-dimensional real semisimple Lie algebra $\mathfr{e}_{k(k)}$ of type $E_{k(k)}$. Adding the generator $K_k$ to the list, noting that $\h_k = \RR K_k \oplus K_k^\perp = \RR K_k \oplus \RR \alpha_1^\vee \oplus \dots \oplus \RR \alpha_k^\vee$, and using the relations for $\g_k$, we obtain the direct sum $\RR K_k \times \mathfr{e}_{k(k)}$, i.e., the trivial extension of $\mathfr{e}_{k(k)}$. Note that in this case the symmetric matrix $C$ and the induced inner product on $K_k^\perp$ are positive definite.

\smallskip 
\noindent (ii) For $k=9$, when the generalized Cartan matrix $C$ has rank 8 and is positive semidefinite, the triple \newline $(\h_9, \{\alpha_i\}_{i=1}^9, \{\alpha_i^\vee \}_{i=1}^9)$ is a realization over $\RR$ of the generalized Cartan matrix $C$ of type $E_{9}$. Therefore, $\g_9$ is the split real affine Kac-Moody algebra of this type. Note that in this case $K_9 \in K_9^\perp = \RR \alpha_1^\vee \oplus \dots \oplus \RR \alpha_9^\vee \subsetneq \h_9$, and the symmetric matrix $C$ and the induced inner product on $K_9^\perp$ are positive semidefinite.

\smallskip 
\noindent (iii) For $k \ge 10$, as in the proof of Part (i), the triple $\big(K_k^\perp, \{\alpha_i\}_{i=1}^k, \{\alpha_i^\vee \}_{i=1}^k\big)$ gives a realization over $\RR$ of the generalized Cartan matrix $C$ of type $E_{k}$. The corresponding real Kac-Moody algebra generated by the Chevalley generators $e_1, \dots, e_{k}, \linebreak[0] f_1, \dots, f_{k}$ and the elements of $K_k^\perp$ is a split real Kac-Moody algebra of type $E_{k(k)}$. It is of indefinite type, because so are the symmetric matrix $C$  and the subspace $K_k^\perp$ in this case. Adding the generator $K_k$ gives the trivial extension of this Kac-Moody algebra.
\end{proof}

\begin{defn}[Lie algebra action on a DGCA]
    An \emph{action of a Lie algebra $\g$ on an augmented DGCA} $(A,d, \eps)$ is understood
as a dg-Lie-algebra homomorphism
\vspace{-3mm} 
  \[
  \g \longrightarrow \Der A,
  \]

\vspace{-2mm} 
\noindent where $\Der A$ is the \emph{Lie algebra of derivations of
  $(A,d, \eps)$}. These are linear maps $A \to A$ of degree 0, commuting with
the differential $d: A \to A$ and augmentation $\eps: A \to \RR$, and satisfying the Leibniz rule.
\end{defn}

\begin{construction}[Parabolic action on toroidification model]
\label{pk-acts}
Consider the \emph{maximal $($proper$)$ parabolic subalgebra} $\p_k \subset \g_k$ generated by all the above generators
but $f_k$ and define an action of $\p_k$ on $M(\mc{T}^k S^4)$
(see Theorem \ref{tor}) as follows:
\vspace{-2mm} 
\begin{equation*}
 e_k g_4 = e_k g_7 = e_k s_i g_4 = e_k s_i g_7 
 = e_k s_i s_j g_7 = e_k s_1 s_2 s_3 g_4 
 = 0, \qquad \qquad \qquad  \quad \; \proofstep{$i, j \in \{1,2,3\}$},
       \end{equation*}
     \vspace{-10mm} 
 \begin{align*}        
 e_k w_i & = 0, & & & \proofstep{$1 \le i \le k$},\\
 e_k s_i s_j g_4 & = \sgn
  \bigl( \begin{smallmatrix}1 & 2 & 3\\ i& j & l\end{smallmatrix}
    \bigr)\, w_l & & & 
       \proofstep{$ \{i, j, l \} = \{1, 2, 3 \}$},\\
          e_k s_1 s_2 s_3 g_7 & = g_4,
           \\
   [e_k, s_i] & = 0,  
& & & \proofstep{$4 \le i \le k$},\\
h g_4 &= - \eps_0(h)
    g_4, & h g_7 & = - 2\eps_0 (h) g_7,\\ 
    [h, s_i] &= \eps_i (h) s_i,
     & h w_i & = - \eps_i(h) w_i, 
     & \proofstep{$1 \le i \le k$},
    \\
    e_i g_4 & = e_i g_7 = 0, 
    &&&\proofstep{$1 \le i \le k-1$},
    \\
    e_i w_j &= \delta_{ij} w_{i+1}, &&&
\proofstep{$1 \le i \le k-1$, $1 \le j \le k$},
\\
               [e_i,s_j] & = - \delta_{i+1,j} s_{i}, &&&
               \proofstep{$1 \le i \le k-1$, $1 \le j \le k$},
               \\
   f_i g_4 &= f_i g_7 = 0, &&&  \proofstep{$1 \le i \le k-1$},
   \\
    f_i w_j &= \delta_{i+1,j} w_{i}, &&&  \proofstep{$1 \le i \le k-1$, $1 \le j \le k$},
    \\
    [f_i,s_j] &= - \delta_{ij} s_{i+1},
    &&&  \proofstep{$1 \le i \le k-1$, $1 \le j \le k$}.
\end{align*}
Here we extend the differentials $s_i$ from the loop space model $M(\mc{L}^k S^4) =
S(U_k^{>0})$ to the toroidification model $M(\mc{T}^k S^4) = S(U_k^{>0})[w_1,\dots, \linebreak[0] w_k] $ by
setting $s_i w_j = 0$ for all $i,j$.
\end{construction}

\begin{remark}[Linearity of action]
Note that when the DGCA $A$ is \emph{semifree}, i.e., free as a graded commutative algebra: $A = S(V)$, 
then a derivation which is induced by a linear map $V \to V \subset S(V)$ is called \emph{linear}. 
If $\g$ acts on $A = S(V)$ by linear derivations, we call such an action \emph{linear}. 
To keep things simple, all actions considered in this paper are linear.
\end{remark}

\begin{theorem}[The parabolic action]
  \label{p-action}
For $k \ge 3$, the above formulas in \Cref{pk-acts} define a
$($linear$)$ action of the parabolic Lie subalgebra
$\p_k$ of $\g_k$ on the Sullivan minimal model $M(\mc{T}^k S^4)$,
\emph{i.e.,} a Lie-algebra homomorphism
\vspace{-2mm} 
  \[
  \p_k \longrightarrow \Der M(\mc{T}^k S^4)\,.
  \]
\end{theorem}

\begin{proof}
The formulas before the theorem, characterizing how the generators of
$\p_k$ act on the generators of $M(\mc{T}^k S^4)$, define graded
derivations of $M(\mc{T}^k S^4)$ as a graded commutative algebra,
because $M(\mc{T}^k S^4)$ is free as such. The action formulas are designed to make sure that these derivations are of degree zero. The action also extends the action of the abelian Lie algebra $\h_k$, the Lie algebra of the maximal split torus from \Cref{maxtorus}. All this pretty much determines the action of the parabolic $\p_k$: a Chevalley generator of $\p_k$ corresponding to a root $\alpha$ acting on a generator of $M(\mc{T}^k S^4)$ of weight $\beta$ maps it to a generator of weight $\alpha + \beta$. Our goal here is to show that the formulas above realize such an action.

The fact that the action respects the differential is
verified in a straightforward, though tedious way, as follows. Since the differential commutes with the action of $\h_k$, it implies that the weight spaces are preserved by the differential. The differential is quadratic and the action is linear, hence, for a Chevalley generator $g$ of weight $\alpha$ and a generator $x \in V \subset S(V) = M = M(\mc{T}^k S^4)$, the elements $dgx$ and $gdx$ will be in the weight space $S^2(V)_{\alpha+\beta}$ and of a particular degree. In some cases, when this subspace is zero, this will automatically imply that $dgx = 0 =  gdx$. For instance, $d e_k s_1 s_2 g_4$ and $e_k d s_1 s_2 g_4$ are in $S^2(V)_{-\eps_3}^3 = 0$. Likewise, $de_k g_4$, $e_k d g_4 \in S^2(V)_{-\eps_1-\eps_2-\eps_3}^5 = 0$, $de_k g_7$, $e_k d g_7 \in S^2(V)_{-\eps_0 -\eps_1-\eps_2-\eps_3}^8 = 0$, $df_i w_j$, $f_i d w_j \in S^2(V)_{\eps_{i+1} -\eps_i-\eps_j}^3 = 0$ for $1 \le i \le k-1$ and $1 \le j \le k$. In other cases, the weight-degree space would be one-dimensional, such as $S^2(V)_{-\eps_2 -\eps_3}^4 = \R w_2 \cdot w_3 $, and this would guarantee that one of the elements $dgx $ and $gdx$ would be a scalar multiple of the other. In this case, we would need a direct verification, such as 
\vspace{-2mm} 
\begin{align*}d e_k s_1 g_4 &= 0 = - w_2 w_3 + w_3 w_2 = - w_2 \cdot e_k s_1 s_2 g_4 - w_3 \cdot e_k s_1 s_3 g_4  
\\
&= - \sum_{l=1}^k w_l \cdot e_k s_1 s_l g_4 = - e_k \sum_{l=1}^k w_l \cdot s_1 s_l g_4 = - e_k s_1 \sum_{l=1}^k w_l \cdot s_l g_4 = -e_k s_1 dg_4 = e_k d s_1 g_4\;.
\end{align*}

\vspace{-2mm} 
\noindent In some other cases, such as $e_k d s_1 s_2 s_3 g_7$, $ d e_k s_1 s_2 s_3 g_7 
\in S^2(V)_{-\eps_0}^5 = \bigoplus_{l=1}^k \R w_l \cdot s_l g_4$, the weight-degree space would be of higher dimension, and again, a direct verification of the equation $dgx = gdx$ would be needed. Here it is in this case:
\vspace{-3mm} 
\begin{align*}
 e_k d s_1 s_2 s_3 g_7 &= \tfrac{1}{2} e_k s_1 s_2 s_3 (g_4^2) -  e_k \sum_{l=1}^k s_1 s_2 s_3( w_l \cdot s_l g_7) 
 \\[-5pt]
 & = e_k (s_2 s_3 g_4 \cdot s_1 g_4 - s_1 s_3 g_4 \cdot s_2 g_4 + s_3 g_4 \cdot s_1 s_2 g_4) + \sum_{l=4}^k  w_l \cdot s_l e_k s_1 s_2 s_3 g_7
 \\[-5pt]
 & =  \sum_{l=1}^k w_l \cdot s_l g_4 = d g_4 = d e_k s_1 s_2 s_3 g_7.
 \end{align*}

 \vspace{-2mm} 
\noindent  Here is one more case like this:
 \vspace{-3mm} 
\begin{align*}
 f_j d s_j g_7 &= f_i s_j g_4 \cdot g_4 - f_i \sum_{l=1}^k s_j s_l g_7 \cdot w_l
 \\[-5pt]
 &= -\delta_{ij} s_{i+1} g_4 \cdot g_4  + \delta_{ij} \sum_{l=1}^k s_{i+1} s_l g_7 \cdot w_l - \sum_{l=1}^k s_j f_i s_l g_7 \cdot w_l - \sum_{l=1}^k s_j s_l g_7 \cdot f_i w_l
 \\[-1pt]
& = \delta_{ij} s_{i+1} d g_7 + s_j s_{i+1} g_7 \cdot w_i - s_j s_{i+1} g_7 \cdot w_i = - \delta_{ij} d s_{i+1} g_7 = d f_i s_j g_7.
 \end{align*}

Having established the compatibility of the action with the differential, we get a linear action of the \emph{free Lie
algebra} $L$ generated by $e_1, e_2, \dots, e_k, f_1, f_2, \linebreak[0]
\dots, \linebreak[1] f_{k-1}$ and the vector space $\h_k$ on the DGCA $M(\mc{T}^k S^4)$.

The rest of the proof is a 
straightforward verification that
the defining relations of the quotient $\p_k$ of $L$ by the ideal generated by the relations (which are those relations of $\g_k$ which only involve the generators of $\p_k$) are satisfied by the operators defining the action. Note that
the relations of commutation with the elements $h \in \h_k$ hold due
to the design of the formulas: the action is defined so that the total
weight of Chevalley generators and generators of the Sullivan algebra
on each side of the formulas is the same. So, the only relations that
need to be checked are the ones that involve the $e_i$'s and
$f_j$'s. These are $[e_i,f_j] = \delta_{ij} \alpha_i^\vee = h_i-h_{i+1}$ for $1 \le i \le k$ and $1 \le j \le k-1$ and the Serre relations. (Recall that $f_k \not\in \p_k$.) Let us first check if the commutator of the actions of $e_i$ and $f_i$ equals to the action of $[e_i,f_i] = \alpha_i^\vee = h_i-h_{i+1}$ for any $i = 1, \dots, k-1$. 
It suffices to check only the action on the generators on which at least one of $e_i$ and $f_i$ acts nontrivially:
\vspace{-2mm} 
\begin{gather*}
w_i  \xmapsto{\;e_i\;}  w_{i+1} \xmapsto{\;f_i\;} w_i, \qquad 
s_{i+1} \xmapsto{\;[e_i, -]\;}  -  s_i  \xmapsto{\;[f_i,-]\;} s_{i+1}, 
\\
w_{i+1} \xmapsto{\;f_i\;} w_i  \xmapsto{\;e_i\;} w_{i+1}, 
\qquad 
s_i \xmapsto{\;[f_i,-]\;}  - s_{i+1} \xmapsto{\;[e_i,-]\;} s_i,
\end{gather*}

\vspace{-2mm} 
\noindent whence
\vspace{-2mm} 
\begin{align*}
\hspace{-6mm} 
e_i(f_i w_i) - f_i (e_i w_i) & = - f_i (e_i w_i)  = - w_i  = - \eps_i (h_i-h_{i+1}) w_i = (h_i-h_{i+1}) w_i, \\
e_i(f_i w_{i+1}) - f_i (e_i w_{i+1}) & = e_i(f_i w_{i+1}) = w_{i+1}  = -\eps_{i+1}(h_i-h_{i+1}) w_{i+1} = (h_i-h_{i+1}) w_{i+1} ,\\
\big[e_i, [f_i,s_i]\big] - \big[f_i, [e_i,s_i]\big]  & = \big[e_i, [f_i,s_i]\big] = s_i  = \eps_i(h_i-h_{i+1}) s_i = [h_i-h_{i+1}, s_i], \\
\big[e_i, [f_i,s_{i+1}]\big] - \big[f_i, [e_i,s_{i+1}]\big] & = - \big[f_i, [e_i,s_{i+1}]\big]  = - s_{i+1}  = \eps_{i+1}(h_i-h_{i+1}) s_{i+1} = [h_i-h_{i+1},s_{i+1}].
\end{align*}
The fact that the commutator of the actions of $e_i$ and $f_j$ for $i \ne j$ equals the action of $[e_i, f_j]$ is verified similarly.

Now consider the Serre relations. For the following Serre relation, we will have to assume that $k > 3$, because for $k=3$ it would just not be imposed.
The relation we would like to check reads
$(\operatorname{ad}\rho(e_k))^{1-c_{k3}}(\rho(e_3)) = 0$ or, in this case, $(\operatorname{ad} \rho(e_k))^2(\rho(e_3)) = [\rho(e_k),[\rho(e_k),\rho(e_3)]] = 0$, where $\rho: L \to \Der^0 M(\mc{T}^k S^4)$ denotes the action of the free Lie algebra $L$ on $M(\mc{T}^k S^4)$. It suffices to see that the operator $[\rho(e_k),[\rho(e_k),\rho(e_3)]]$ acts on every generator of $M(\mc{T}^k S^4)$ by zero. Note that the weight of the operator is $2 \eps_0 - 2 \eps_1 - 2 \eps_2 - \eps_3 -\eps_4$. Now observe, by analyzing the weights of the generators given by \Cref{tor}, that no two generators of $M(\mc{T}^k S^4)$ differ by this weight. Since the action $\rho$ is linear, i.e., takes generators to linear combinations thereof, this implies that the operator $[\rho(e_k),[\rho(e_k),\rho(e_3)]]$ vanishes on every generator of $M(\mc{T}^k S^4)$.

 The other Serre relations corresponding to $c_{ij} = -1$ are checked similarly. The simpler Serre relations corresponding to $c_{ij} = 0$ are checked directly, just like the relation $[e_i,f_i] = \alpha_i^\vee$ above. For example, to check the relation $[e_1, e_3] = 0$, we look at those generators of $M(\mc{T}^k S^4)$ on which either $e_1$ or $e_3$ acts nontrivially:
 \vspace{-3mm} 
 \begin{align*}
 [e_1, [e_3, s_2]] - [e_3, [e_1, s_2] & = [e_1, 0 ] + [e_3, s_1] = 0,\\
 [e_1, [e_3, s_4]] - [e_3, [e_1, s_4] & = -[e_1, s_3] - [e_3, 0] = 0,\\
 e_1(e_3 w_1) - e_3 (e_1 w_1) & = e_1 \cdot 0 - e_3 w_2  = 0,\\
  e_1(e_3 w_1) - e_3 (e_1 w_1) & = e_1 \cdot 0 - e_3 w_2  = 0.
\end{align*}
To check the relation $[e_1, e_k] = 0$, we use
\vspace{-2mm}
 \begin{align*}
 e_1 (e_k s_2 s_3 g_4) - e_k (e_1 s_2 s_3 g_4) & = e_1 w_1 + e_k s_1 s_3 g_4 = w_2 - w_2 = 0,\\
 e_1 (e_k s_1 s_2 s_3 g_7) - e_k (e_1 s_1 s_2 s_3 g_7) & = e_1 g_4 + e_k s_1 s_1 s_3 g_7 =0 + 0 = 0. \qedhere
\end{align*}
\end{proof}

\smallskip 
\begin{remark}[Action of the gravity line]
\label{gravity-line}
    Since the ``gravity line'' Lie algebra $\mathfrak{sl}(k,\R)$ (see \Cref{sec-para}) may be identified with the subalgebra of $\p_k$ generated by $e_1, \dots, e_{k-1}$, $f_1, \dots, f_{k-1}$, and $\alpha_1^\vee, \dots, \alpha_{k-1}^\vee$, \Cref{p-action} gives, in particular, an action of the ``gravity line'' Lie algebra on $M(\mc{T}^k S^4)$. This works for all $k \ge 0$, as the cases $0 \le k \le 2$ are treated in the next section. Note that the restriction of the action $\g_k \to \Der M(\mc{T}^k S^4)$ to $\mathfrak{sl}(k,\R)$ is \emph{faithful}, i.e., the kernel of the action homomorphism $\mathfrak{sl}(k,\R) \to \Der M(\mc{T}^k S^4)$ is zero, because this action is nontrivial by construction and the Lie algebra $\mathfrak{sl}(k,\R)$ is simple.
\end{remark}

\section{The small \texorpdfstring{$k$}{} case}
\label{small-k}

Here we complement the action of the parabolic $\p_k$ on $M(\mc{T}^k S^4)$ defined in the previous section \cref{parabolic_action} for $k \ge 3$ with the small $k$ case: $0 \le k \le 2$, which does not quite follow the same pattern. In this case, the real Lie algebra $\g_k$ will be defined by starting with its Cartan subalgebra $\h_k$ with a basis $\{ h_0, \dots, h_k\}$, just as in \eqref{basis-h}, with the same Minkowski inner product and a basis $\{ \eps_0, \dots, \eps_k\}$ of $\h_k^*$, also as introduced right after \eqref{basis-h}.
The difference with the case $k \ge 3$ begins with the set of simple roots. Here, we will just have
\vspace{-2mm}
\[
\alpha_1 \coeq \eps_{1} - \eps_{2}  \qquad \text{for $k = 2$}
\]
as the only simple root (and, respectively, the only simple coroot $\alpha_1^\vee \coeq h_{1} - h_{2}$) and no roots or coroots whatsoever for $k = 0$ and $1$. Since the simple root $\alpha_k$ in the Dynkin diagram $E_k$ for $0 \le k \le 2$, see \eqref{Dynkin-schem-3} and the paragraph after that, is absent, it makes sense to define the parabolic subalgebra corresponding to $\alpha_k$ as $\p_k = \g_k$.

\medskip 
The full Lie algebra $\g_2$, by analogy with how it is defined for $k \ge 3$ just before \Cref{identification} will be generated by the real vector space $\h_2$ spanned by $h_0$, $h_1$, and $h_2$, and the elements $e_1$ and $f_1$, such that
\vspace{-1mm}
\begin{align*}
[h, h'] & = 0 && \text{for $h, h' \in \h_2$,}\\
[h,e_1] & = \alpha_1(h) e_1  && \text{for $h \in \h_2$,}\\
[h,f_1] & = - \alpha_1(h) f_1  & & \text{for $h \in \h_2$,}\\
[e_1, f_1] & = \alpha_1^\vee.
\end{align*}
Here we recognize the Lie subalgebra $\mathfr{sl}(2,\R)$ generated by $\alpha_1^\vee = h_1-h_2$, $e_1$, and $f_1$, which sits in $\g_2$ as a direct summand:
\vspace{-1mm}
\[
\g_2 \cong \mathfr{sl}(2,\R) \times (h_1-h_2)^\perp \cong \mathfr{gl}(2,\R) \times \R,
\]
where $(h_1-h_2)^\perp \subset \h_2$ is the abelian Lie algebra generated by $h_0$ and $h_1 + h_2$.

\medskip 
For $k=1$ and $0$, the situation is simpler:
\vspace{-3mm}
\begin{align*}
\g_1 & = \R h_0 \oplus \R h_1,\\
\g_0 & = \R h_0.
\end{align*}

Since the parabolic for $0 \le k \le 2$ is just the full Lie algebra $\g_k$, we will define its action on $M(\mc{T}^k S^4)$ essentially by the same formulas as in \Cref{pk-acts}, which simplify to the following:
\vspace{-2mm} 
\begin{align*}
h g_4 &= - \eps_0(h)
    g_4, & h g_7 & = - 2\eps_0 (h) g_7,\\ 
    [h, s_i] &= \eps_i (h) s_i,
     & h w_i & = - \eps_i(h) w_i, 
     & \proofstep{$1 \le i \le k \le 2$},
    \\
    e_1 g_4 & = e_1 g_7 = 0, 
    &&&\proofstep{$ k = 2$},
    \\
    e_1 w_1 &=  w_{2}, & e_1 w_2 & = 0, &\proofstep{$ k = 2$},
\\
               [e_1,s_1] & = 0, & 
               [e_1,s_2] & = - s_{1}, & \proofstep{$k = 2$},
               \\
   f_1 g_4 &= f_1 g_7 = 0, &&&  \proofstep{$k=2$},
   \\
    f_1 w_1 &= 0, & f_1 w_2 & = w_1 &  \proofstep{$k = 2$},
    \\
    [f_1,s_1] &= - s_{2}, & [f_1,s_2] &= 0
    &  \proofstep{$k= 2$.}
\end{align*}

The following theorem is verified the same way as \Cref{p-action}, but there are just fewer identities to check.
\begin{theorem}[The action for $k \le 2$]
  \label{g-action}
For $0 \le k \le 2$, the above formulas define a
$($linear$)$ action of the Lie algebra $\g_k$ on the Sullivan minimal model $M(\mc{T}^k S^4)$,
\emph{i.e.,} a Lie-algebra homomorphism
\vspace{-2mm} 
  \[
  \g_k \longrightarrow \Der M(\mc{T}^k S^4)\,.
  \]
\end{theorem}

One can easily check by hand, using the properties $M(S^4)^4 = \R g_4$ and $M(S^4)^7 = \R g_7$, that $\Der M (S^4) = \g_0$ and $\Aut M(S^4) = \GG_m(\R)$. A similar simple computation shows that $\Aut_l M(\mc{T}^1 S^4) \linebreak[0] = \GG_m(\R)^2$ and $\Der_l M(\mc{T}^1 S^4) = \g_1$, where $\Aut_l$ and $\Der_l$ denote the group of \emph{linear} automorphisms and Lie algebra of \emph{linear} derivations of the DGCA $M(\mc{T}^1 S^4) = M(\mc{L}_c S^4)$, respectively. The notion of linearity depends on the presentation of the Sullivan minimal model as $S(V)$. A much more elaborate but still straightforward computation shows that the full Lie algebra $\Der M(\mc{T}^1 S^4)$ of derivations is five-dimensional. Similar methods show that $\Der_l M(\mc{T}^2 S^4) = \g_2$. We do not know if this is the case for higher $M(\mc{T}^k S^4)$.

\medskip 
\noindent  Hisham Sati, {\it Mathematics, Division of Science, and 
\\
\indent Center for Quantum and Topological Systems (CQTS),  NYUAD Research Institute, 
\\
\indent New York University Abu Dhabi, UAE, and 
\\
\indent The Courant Institute for Mathematical Sciences, NYU, NY, USA.}
\\
\indent {\tt hsati@nyu.edu}
\\
\\
\noindent  Alexander A. Voronov, {\it School of Mathematics, University of Minnesota, Minneapolis, MN 55455, USA, and
\\
\indent Kavli IPMU (WPI), UTIAS, University of Tokyo, Kashiwa, Chiba 277-8583, Japan.}
\\
\indent {\tt voronov@umn.edu}

\end{document}